\documentclass{amsart}
\usepackage[a4paper,centering,textwidth=145mm,textheight=195mm]{geometry}

\DeclareFontFamily{U}{mathx}{}
\DeclareFontShape{U}{mathx}{m}{n}{<-> mathx10}{}
\DeclareSymbolFont{mathx}{U}{mathx}{m}{n}
\DeclareMathAccent{\widehat}{0}{mathx}{"70}
\DeclareMathAccent{\widecheck}{0}{mathx}{"71}

\numberwithin{equation}{section}

\title[Octagons of Witt groups]{Categories with conjugation\\ and
  octagons of Witt groups}
\author[D.W. Lewis \and J.-P. Tignol]{David W. Lewis$^{\dag}$ \and
  Jean-Pierre Tignol} 

\address[Jean-Pierre Tignol]{%
UCLouvain, ICTEAM Institute\\ B-1348 Louvain-la-Neuve, Belgium
}  

\email{jean-pierre.tignol@uclouvain.be}

\thanks{\dag Author deceased. The second author gratefully
  acknowledges the hospitality of 
  the first author and UCD Dublin when work for this paper
  was initiated. He was
  supported in part by the National Fund for Scientific Research
  (Belgium) and by the European 
  Community’s Human Potential Programme under contract
  HPRN-CT-2002-00287, KTAGS}  


\usepackage[all]{xypic}
\xyoption{2cell}
\UseAllTwocells
\newdir{ >}{{}*!/-12pt/@{>}}
\usepackage{BOONDOX-cal}
\newcommand{\catA}{\mathcal{A}}
\newcommand{\catB}{\mathcal{B}}
\newcommand{\catC}{\mathcal{C}}
\newcommand{\catD}{\mathcal{D}}
\newcommand{\catE}{\mathcal{E}}
\newcommand{\catF}{\mathcal{F}}
\newcommand{\catH}{\mathcal{H}}
\newcommand{\catM}{\mathcal{M}}
\newcommand{\catP}{\mathcal{P}}
\newcommand{\catS}{\mathcal{S}}
\newcommand{\Idnt}{\mathbb{I}}
\newcommand{\Idfunc}{\operatorname{Id}}

\newcommand{\conjt}{\iota}
\newcommand{\llambda}{\lambda}
\newcommand{\god}{\cdot}
\newcommand{\isom}{\xrightarrow{\sim}}
\newcommand{\natiso}{\stackrel{\sim}{\Rightarrow}}

\newcommand{\ext}{\mathsf{ext}}
\newcommand{\tr}{\mathsf{tr}}
\newcommand{\can}{\mathsf{can}}
\newcommand{\Mor}{\mathsf{M}}
\newcommand{\rad}{R}

\renewcommand{\circ}{\,}

\DeclareMathOperator{\Hom}{Hom}
\DeclareMathOperator{\End}{End}

\DeclareMathOperator{\Int}{Int}

\newtheorem{proposition}{Proposition}[section]
\newtheorem{lemma}[proposition]{Lemma}
\newtheorem{theorem}[proposition]{Theorem}

\theoremstyle{remark}
\newtheorem{remark}[proposition]{Remark}

\theoremstyle{definition}
\newtheorem*{property}{Property (P)}

\begin{document}

\begin{abstract}
A categorical approach is proposed for the construction of
$8$-periodic chain complexes of Witt groups of categories with
duality. Exactness of the 
resulting sequences is proved under the hypothesis that
the categories are abelian and artinian, or when the endomorphism
algebra of each object is semilocal and rad-adically complete.
\end{abstract}

\maketitle

\section*{Introduction}

In the algebraic theory of quadratic forms, the special r\^ole played
by 
quadratic field extensions was noticed early. Among the first
observations stands out the description by
Pfister~\cite[p.~123]{Pfister} of quadratic forms that become 
hyperbolic over a given quadratic extension, reformulated in terms of
hermitian forms by Milnor--Husemoller: see~\cite[App.~2]{MH}, where it
appears as a $4$-term exact sequence. Using Scharlau's transfer,
Elman--Lam prolonged the sequence into a long $3$-periodic exact
sequence (an ``exact triangle'' or ``symmetric exact hexagon'',
see~\cite[Th.~2.6]{EL}). 
Extending this
early result in the framework of hermitian or skew-hermitian forms
over quaternion algebras, the first author uncovered a more complex
behaviour, giving rise to longer exact sequences relating forms over
quaternion algebras and over quadratic subfields, see~\cite{L1},
\cite{L2} (see also
\cite[Th.~10.3.2]{Scharlau}). It was soon realised that exact
sequences thus 
constructed are $8$-periodic, and can therefore be folded into exact
octagons,
see~\cite{L3}. Meanwhile, Ranicki also constructed $8$-periodic exact
sequences of Witt groups for ``twisted quadratic extensions'' via
$L$-theory: see~\cite{R} and \cite[(0.2)]{HTW}.

Building on the first author's work,
Parimala--Sridharan--Suresh~\cite[Appendix~2]{BFP} 
extended 
the scope of~\cite{L2} beyond the case of quaternion algebras,
obtaining exact sequences relating Witt groups of hermitian or
skew-hermitian forms over a central simple algebra and over the
centraliser of a quadratic subfield. These exact sequences were
further patched into exact octagons of Witt groups by
Grenier-Boley--Mahmoudi~\cite{GBM}. The latest development in this
line of investigation is due to First~\cite{First}, who obtained an
analogous exact octagon of Witt groups for Azumaya algebras over
semilocal rings. His paper also discusses several important
applications.
\medbreak

The purpose of this work is to describe a categorical mechanism that
produces the octagons referred to above. Our approach is close in
spirit to the method in~\cite{L4} making use of the periodicity of the
Clifford algebra construction and Morita equivalences. The basic tool
is a notion of 
\emph{conjugation} on an additive pseudo-abelian (=
idempotent-complete) category with $\frac12$ (i.e., such that $2$ is
invertible in each $\Hom$ group). A conjugation is a covariant version
of a duality, 
consisting of an endofunctor $\conjt$ and a natural isomorphism
$\lambda\colon \conjt\conjt\natiso\Idfunc$. The construction of the
hermitian category of symmetric spaces from a category with duality
has an analogue for a category with conjugation
$(\catC,\conjt,\lambda)$, which yields a \emph{double} category with
conjugation $(\catC',\conjt',\lambda')$. The categories $\catC$ and
$\catC'$ are related by a pair of functors $F\colon \catC\to\catC'$
and $G\colon\catC'\to\catC$. The functor $F$ is reminiscent of the
construction of hyperbolic symmetric spaces in the context of
categories with duality, while $G$ is a forgetful functor. Iterating
the doubling construction yields an infinite sequence of categories
with conjugation $(\catC'',\conjt'',\lambda'')$, $(\catC''',
\conjt''',\lambda''')$, \ldots\ and functors
\[
  \xymatrix{
    \catC\ar[r]^{F}&
    \catC'\ar[r]^{F'}&
    \catC''\ar[r]^{F''}
    &\cdots
  }
  \qquad\text{and}\qquad
  \xymatrix{
    \catC&
    \catC'\ar[l]_{G}
    &
    \catC''\ar[l]_{G'}
    &
    \cdots\ar[l]_{G''}
  }
\]
We show in
\S\ref{subsec:periodconj} that this construction is $2$-periodic by
defining an equivalence
$(\catC,\conjt,\lambda)\equiv(\catC'',\conjt'',\lambda'')$, which
allows us to fold the sequence above into a $2$-gon
$\xymatrix{
\catC
\ar@<0.5ex>[r]^{F}
&
\catC'
\ar@<0.5ex>[l]^{G}}$.

When $\catC$ also carries a
duality $(D,\delta)$, commuting in a natural sense with a conjugation
$(\conjt,\lambda)$, a ``conjugate'' duality $(\conjt D,\delta_\alpha)$
on $\catC$ and ``double'' dualities $(D',\delta')$ and
$(\conjt'D',\delta'_{\alpha'})$ on $\catC'$ are constructed in
\S\ref{sec:dual}. In contrast with the doubling construction, the
resulting categories with duality form an $8$-periodic sequence,
see~\eqref{eq:octacatdual}:
\begin{equation}
  \label{eq:intro}
  \begin{split}
  \xymatrix{
    (\catC,D,\delta) \ar[r]^-{\widehat F}
    &
    (\catC',D',\delta')\ar[r]^-{\widehat G}
    &
    (\catC,\conjt D,-\delta_\alpha) \ar[r]^-{\widecheck F}
    &
    (\catC',\conjt' D',-\delta'_{\alpha'}) \ar[d]^{\widecheck G}
    \\
    (\catC',\conjt' D',\delta'_{\alpha'}) \ar[u]^{\widecheck G}
    &
    (\catC,\conjt D,\delta_\alpha) \ar[l]_-{\widecheck F}
    &
    (\catC',D',-\delta') \ar[l]_-{\widehat G}
    &
    (\catC,D,-\delta) \ar[l]_-{\widehat F}
  }
\end{split}
\end{equation}

The corresponding hermitian categories of symmetric spaces discussed
in~\S\ref{sec:symsp} are 
similarly organised, see~\eqref{eq:octasspaces}, and we show in
Theorem~\ref{thm:octamain} that under the composition of two
consecutive functors from the $8$-periodic sequence the image of each
hermitian category consists of hyperbolic symmetric spaces.

Defining a Witt group from the hermitian category of a category with
duality $(\catC,D,\delta)$ requires an exact structure on $\catC$,
which provides the definition of metabolic symmetric spaces,
see~\cite[\S1.2.4]{Balmer}. We consider two special cases
in~\S\ref{sec:Witt}. For the \emph{split exact structure}, in 
which exact sequences split, metabolic symmetric spaces are
hyperbolic, and Theorem~\ref{thm:octamain} readily yields
from~\eqref{eq:intro} an $8$-periodic chain complex. Exactness
requires a kind of Witt cancellation property (see
Remark~\ref{rem:WCP}), which we establish under the hypothesis that
the endomorphism algebra of every object in $\catC$ is semilocal and
rad-adically complete: see Theorem~\ref{thm:octasplit}. This theorem does
\emph{not} cover the case discussed by First in~\cite{First}, which
deals with 
the split exact category of finitely generated projective modules over
Azumaya algebras over semilocal rings, but it yields the
Grenier-Boley--Mahmoudi octagon \cite{GBM} as a special case.

When $\catC$ is an abelian category, with its usual exact
structure, our main result establishes
the exactness of the 
octagon of Witt groups derived from~\eqref{eq:intro} under the
hypothesis that the abelian category $\catC$ is artinian: see
Theorem~\ref{thm:abelcat}. 
This result is sufficient to account for the exact octagon
from~\cite{L3}, and also for Warshauer's exact octagon
of Witt groups of asymmetric spaces in~\cite{War}. By contrast, since
characteristic~$2$ is ruled out from the outset in this work, our
results do not cover the version of \cite{L2} in
characteristic~$2$ given by Knus--Villa in~\cite[Th.~3.8]{KV}.

The sequence of categories with duality $(\catC,D,\delta)$, $(\catC,
\conjt D, -\delta_\alpha)$, $(\catC,D,-\delta)$, $(\catC,\conjt
D,\delta_\alpha)$  extracted from the $8$-periodic
sequence~\eqref{eq:intro} bears some formal resemblance with the
sequence of shifted dualities $T^n(\catC, D,\delta)$ of a triangulated
category with duality, which also satisfies $T^{n+2}(\catC, D,\delta)
\equiv T^n(\catC, D, -\delta)$ for all $n\in\mathbb{Z}$ and is
therefore 
$4$-periodic, see~\cite[63--65]{Balmer}. However, a triangulation on
an additive category is a very different feature than a conjugation,
as witnessed by the category $\catP_A$ of finitely generated right
modules over an artinian simple ring $A$ with involution, with the
usual duality $(D,\delta)$ defined in~\S\ref{subsec:exsymsp1} below
and the split exact structure: it is easy to find examples of
conjugations commuting with $(D,\delta)$ such that all the Witt groups
obtained from the sequence $(\catP_A,D,\delta)$, $(\catP_A,\conjt
D,-\delta_\alpha)$, $(\catP_A,D,-\delta)$, $(\catP_A,\conjt
D,\delta_\alpha)$ (which are Witt groups of hermitian spaces) are
nonzero, whereas the odd-indexed derived Witt groups vanish,
see~\cite[Th.~4.2]{BP}.

The examples of octagons worked out in detail in this work are all
based on categories of modules, but even in this particular case they
do not exhaust the full range of possibilities (see
Remark~\ref{rem:bimod}). The formalism of conjugations extends
much farther; it will hopefully be useful in geometric settings for
more general exact categories with duality.

\subsection*{Notation}

For any category $\catC$, the notation $C\in\catC$ means that $C$ is
an object of $\catC$.
Natural transformations of functors are denoted by double arrows.
If $F$, $F'$ are functors $\catC\to\catD$ and $G$, $G'$ are functors
$\catD\to\catE$ , we write simply $GF$, $G'F'$ for the composition of
functors.
If $a\colon
F\Rightarrow F'$ and 
$b\colon G\Rightarrow G'$ are natural transformations, we write
$b\god a$  for the Godement product (horizontal composition)
$GF\Rightarrow G'F'$; thus, for 
every $C\in\catC$ the morphism $(b\god a)_C\colon GF(C)\to G'F'(C)$ is
the diagonal of the following commutative square (assuming $G$, $G'$
are covariant):
\[
  \xymatrix@C+5pt{
  GF(C)\ar[r]^{G(a_{C})}\ar[d]_{b_{F(C)}}&GF'(C)\ar[d]^{b_{F'(C)}}\\
  G'F(C)\ar[r]^{G'(a_{C})}&G'F'(C)
}
\]
(If $G$, $G'$ are contravariant, the arrows $G(a_C)$ and $G'(a_C)$ are
reversed.) 
In particular, denoting by $\Idnt_{F}\colon\,F\natiso F$ and
$\Idnt_{G}\colon\,G\natiso G$ the identity natural
isomorphisms, the 
natural transformations $\Idnt_{G}\god a\colon\,G F\Rightarrow
G F'$ 
and $b\god\Idnt_{F}\colon\,G F\Rightarrow G' F$ are 
defined by
\[
(\Idnt_{G}\god a)_{C}=G(a_{C}),\qquad
(b\god\Idnt_{F})_{C}=b_{F(C)}.
\]

\section{Categories with conjugation}
\label{sec:double}

Throughout this section, $\catC$ denotes a $k$-linear category for an
arbitrary commutative ring $k$, and all the functors are tacitly
assumed to be $k$-linear. Although this is not necessary for the
very first definitions, we assume $2$ is invertible in $k$ and $\catC$
is pseudo-abelian (=idempotent-complete, see
\cite[Ch.~I, Prop.~6.9]{Karoubi}), which means that every 
idempotent morphism splits. The latter
property holds for instance for the category of finitely-generated
projective modules over an arbitrary $k$-algebra.

\subsection{Definitions}
\label{subsec:def}
A \emph{conjugation} on a category $\catC$ is a pair $(\conjt,\lambda)$
consisting of a $k$-linear covariant functor
$
  \conjt\colon\,\catC\to\catC
$
and a natural isomorphism
$
  \lambda\colon\,\conjt\conjt\natiso\Idfunc_{\catC}
$
such that $\lambda\god\Idnt_{\conjt}=\Idnt_{\conjt}\god\lambda$, i.e.\ for every
object $C$ in $\catC$,
\[
\lambda_{\conjt(C)}=\conjt(\lambda_{C})\colon\,
\conjt\conjt\conjt(C)\isom
\conjt(C).
\]
A \emph{morphism} of categories with conjugation
$(F,\tau)\colon\,(\catC_1,\conjt_1,\lambda_1)\to
(\catC_2,\conjt_2,\lambda_2)$ consists of 
a $k$-linear covariant functor $F\colon\,\catC_1\to\catC_2$ and a
natural isomorphism 
$\tau\colon\,F\conjt_1\natiso\conjt_2F$ such that the following
diagram commutes:
\[
\xymatrix{
\conjt_2F\conjt_1
\ar@{=>}[r]^{\Idnt_{\conjt_2}\god\tau}
\ar@{=>}[d]_{\tau^{-1}\god\Idnt_{\conjt_1}}
&
\conjt_2\conjt_2F
\ar@{=>}[d]^{\lambda_2\god\Idnt_{F}}
\\
F\conjt_1\conjt_1
\ar@{=>}[r]_{\Idnt_{F}\god\lambda_1}
&
F
}
\]
\medbreak

Given a conjugation $(\conjt,\lambda)$ on a category $\catC$ we 
define the \emph{double} category $\catC'$: its objects are pairs
$(C,f)$ where 
$C\in\catC$ and $f\colon\,\conjt(C)\to C$ is an
isomorphism in $\catC$ such that the following diagram commutes:
\[
\xymatrix{
\conjt{\conjt(C)}
\ar[r]^{\conjt(f)}
\ar[dr]_{\lambda_{C}}
&
\conjt(C)
\ar[d]^{f}
\\
&
C
}
\]
(Note the analogy with the definition of a symmetric space in a
category with duality; see \S\ref{sec:symsp}.)
Morphisms $(C,f)\to(D,g)$ in $\catC'$ are morphisms $\varphi\colon\,C\to
D$ in $\catC$ such that $\varphi\circ f=g\circ\conjt(\varphi)$.

Clearly, the double category $\catC'$ is $k$-linear, with direct sum
\[
(C_1,f_1)\oplus(C_2,f_2)=\bigl(C_1\oplus C_2,
\bigl(
\begin{smallmatrix}
f_1&0\\
0&f_2
\end{smallmatrix}
\bigr)\bigr).
\]
It is also pseudo-abelian, because idempotent morphisms $e\colon
(C,f)\to (C,f)$ in $\catC'$ are also idempotent in $\catC$, and if
$K\xrightarrow{t}C\xrightarrow{s}K$ are morphisms in $\catC$ such that
$s\circ t=\Idfunc_K$
and
$t\circ s=e$, then $\bigl(K,s\circ f\circ\conjt(t)\bigr)
\xrightarrow{t}(C,f)\xrightarrow{s}\bigl(K,s\circ
f\circ\conjt(t)\bigr)$ are morphisms in $\catC'$ that split
$e$.
\medbreak

The conjugation $(\conjt,\lambda)$ on $\catC$ induces a
conjugation $(\conjt',\lambda')$ on $\catC'$ as follows: for $(C,f)$,
$(D,g)\in\catC'$ and $\varphi\in\Hom_{\catC'}\bigl((C,f),(D,g)\bigr)$,
\[
  \conjt'(C,f)=\bigl(\conjt(C),-\conjt(f)\bigr), \qquad
  \conjt'(\varphi)=\conjt(\varphi) \quad\text{and}\quad
  \lambda'_{(C,f)}=-\lambda_C.
\]
(Note the signs in the definitions of $\iota'(C,f)$ and of
$\lambda'_{(C,f)}$.) 
\medbreak

The category with conjugation $(\catC,\conjt,\lambda)$ and its double
$(\catC',\conjt',\lambda')$ are related by a pair of
adjoint\footnote{The functors $F$ and $G$ are adjoint to each other on
  each side. We omit the easy proof because this property is not used
  in the sequel.}
functors $F$, $G$, which are described next.

For $C\in\catC$, let
\[
\Lambda_{C}=
\begin{pmatrix}
0&\lambda_{C}\\
\Idfunc_{\conjt(C)}&0
\end{pmatrix}\colon\,
\conjt(C)\oplus\conjt\conjt(C)\to C\oplus\conjt(C).
\]
Then $(C\oplus\conjt(C),\Lambda_{C})\in\catC'$ and we may
define a $k$-linear functor
\[
F\colon\,\catC\to\catC'
\quad\text{by}\quad
C\mapsto(C\oplus\conjt(C),\Lambda_{C}),\quad
\varphi\mapsto
\begin{pmatrix}
\varphi&0\\
0&\conjt(\varphi)
\end{pmatrix}.
\]

On the other hand, we consider the forgetful functor
\[
G\colon\,\catC'\to\catC\qquad
(C,f)\mapsto C,\quad
\varphi\mapsto\varphi.
\]

\begin{proposition}
  \label{prop:FG}
  There are natural isomorphisms $\Idfunc_{\catC}\oplus\conjt \natiso
  GF$ and $\Idfunc_{\catC'}\oplus\conjt'\natiso FG$.
\end{proposition}

\begin{proof}
  The first isomorphism is clear since $GF(C)=C\oplus\conjt(C)$ for
  $C\in\catC$. For $C'=(C,f)\in\catC'$, it is readily verified that
  $\frac12\bigl(
  \begin{smallmatrix}
    \Idfunc_C&-f\\ f^{-1}&\Idfunc_{\conjt(C)}
  \end{smallmatrix}
  \bigr)$ is an isomorphism $C'\oplus\conjt'(C')\isom FG(C')$.
\end{proof}

\subsection{Periodicity}
\label{subsec:periodconj}

Let $(\catC,\conjt,\lambda)$ be a category with conjugation and
$(\catC',\conjt',\lambda')$ its double. Repeating the doubling
construction, we obtain a new category with conjugation
$(\catC'',\conjt'',\lambda'')$ 
and, by induction, an infinite sequence of categories with
conjugations related by functors defined on the same model as $F$ and
$G$ in~\S\ref{subsec:def}:
\[
\xymatrix{
\catC
\ar@<0.5ex>[r]^{F}
&
\catC'
\ar@<0.5ex>[r]^{F'}
\ar@<0.5ex>[l]^{G}
&
\catC''
\ar@<0.5ex>[r]^{F''}
\ar@<0.5ex>[l]^{G'}
&
\catC'''
\ar@<0.5ex>[r]^{F'''}
\ar@<0.5ex>[l]^{G''}
&
\cdots
\ar@<0.5ex>[l]^{G'''}
}
\]
We show in this subsection that this sequence is
$2$-periodic.
\medbreak

We first define a pair of $k$-linear functors
$\xymatrix{\catC''
  \ar@<0.5ex>[r]^{\Theta}
  &
\ar@<0.5ex>[l]^{\Psi}
\catC}$.
To simplify
notation, we write $(C,f,g)$ for
$C''=\bigl((C,f),g\bigr)\in\catC''$. With 
this notation, the map $g$ is a morphism $\conjt'(C,f)\to(C,f)$ in
$\catC'$, hence
\begin{equation}
\label{fgantic.eq}
g\circ(-\conjt(f))=f\circ\conjt(g).
\end{equation}
Moreover, $g\circ\conjt'(g)=\lambda'_{(C,f)}$, hence
\begin{equation}
\label{f2g2.eq}
g\circ\conjt(g)=-\lambda_{C}=-f\circ\conjt(f).
\end{equation} 
It follows that
\[
  g\circ f^{-1}\circ g\circ f^{-1}=
  -g\circ\conjt(g)\circ\conjt(f)^{-1}\circ
f^{-1}=
\Idfunc_{C}.
\]
Therefore, $\frac12(\Idfunc_{C}-g\circ f^{-1})$ is an idempotent in
$\End_{\catC}(C)$. Since $\catC$ is assumed to be pseudo-abelian,
there are morphisms
$\Theta(C'')\xrightarrow{t_{C''}}C\xrightarrow{s_{C''}} \Theta(C'')$
in $\catC$
such that
\begin{equation}
  \label{second.eq}
s_{C''}\circ t_{C''}=\Idfunc_{\Theta(C'')}
\quad\text{and}\quad
t_{C''}\circ s_{C''}={\textstyle\frac12}(\Idfunc_{C}-g\circ f^{-1}).
\end{equation}
Thus, $\Theta(C'')$ is ``the'' kernel of
$\frac12(\Idfunc_C+g\circ f^{-1})$. 
Define a functor
\[
\Theta\colon\,\catC''\to\catC
\quad\text{by}\quad
\begin{cases}
C''\mapsto \Theta(C''),\\
(\varphi\colon C_1''\to C_2'')\mapsto
\bigl(s_{C_2''}\circ\varphi\circ
t_{C_1''}\colon\Theta(C_1'')\to\Theta(C_2'')\bigr). 
\end{cases}
\]
Note that the morphisms $\conjt(t_{C''})$, $\conjt(s_{C''})$
satisfy
\[
  \conjt(s_{C''})\circ\conjt(t_{C''})=\Idfunc_{\conjt(\Theta(C''))}
  \qquad\text{and}\qquad
  \conjt(t_{C''})\circ\conjt(s_{C''})=
  \textstyle{\frac12}
  (\Idfunc_{\conjt(C)}-\conjt(g)\circ\conjt(f)^{-1}),
\]
so we may identify $\conjt(\Theta(C''))$ with
$\Theta(\conjt(C),-\conjt(f),-\conjt(g))=\Theta\conjt''(C'')$ to
obtain a natural 
isomorphism $\theta\colon \Theta\conjt''\natiso\conjt \Theta$,
hence a 
morphism of categories with conjugation
\[
(\Theta,\theta)\colon(\catC'',\conjt'',\lambda'')\to
(\catC,\conjt,\lambda).  
\]

On the other hand, for every $C\in\catC$, define
\[
\Delta_{C}=\begin{pmatrix}
0&-\lambda_{C}\\
\Idfunc_{\conjt(C)}&0
\end{pmatrix}
\colon\,\conjt(C)\oplus\conjt\conjt(C)\to
C\oplus\conjt(C).
\]
Computation shows that
$-\Delta_{C}\circ\conjt(\Lambda_{C})=
\Lambda_{C}\circ\conjt(\Delta_{C})$, hence 
$\Delta_{C}$ defines a morphism in $\catC'$,
\[
\Delta_{C}\colon\,\conjt'(C\oplus\conjt(C),\Lambda_{C})\to
(C\oplus\conjt(C),\Lambda_{C}). 
\]
Moreover,
\[
\Delta_{C}\circ\conjt'(\Delta_{C})=
\begin{pmatrix}
-\lambda_{C}&0\\
0&-\lambda_{\conjt(C)}
\end{pmatrix}
=\lambda'_{(C\oplus\conjt(C),\Lambda_{C})},
\]
hence $(C\oplus\conjt(C),\Lambda_{C},\Delta_{C})\in\catC''$.
Define a functor
\[
\Psi\colon\,\catC\to\catC''
\quad\text{by}\quad
C\mapsto(C\oplus\conjt(C),\Lambda_{C},\Delta_{C}),
\quad
\varphi\mapsto
\begin{pmatrix}
\varphi&0\\
0&\conjt(\varphi)
\end{pmatrix}.
\]
For every $C\in\catC$, let
\[
\psi_C=
\begin{pmatrix}
-\Idfunc_{\conjt(C)}&0\\
0&\Idfunc_{\conjt\conjt(C)}
\end{pmatrix}\colon
\conjt(C)\oplus\conjt\conjt(C)\to
\conjt(C)\oplus\conjt\conjt(C).
\]
The map $\psi_C$ defines an isomorphism in $\catC''$
\[
  \psi_C\colon \Psi\conjt(C)\isom
\conjt''\Psi(C),
\]
hence a natural isomorphism $\psi\colon
\Psi\conjt\natiso\conjt''\Psi$. The equation
$\lambda''_{\Psi(C)}\circ \conjt''(\psi_C) = \Psi(\lambda_C)\circ
\psi^{-1}_{\conjt(C)}$ is readily verified for every $C\in\catC$,
hence $(\Psi,\psi)$ is a morphism of categories with conjugation
\[
  (\Psi,\psi)\colon (\catC,\conjt,\lambda)\to
  (\catC'',\conjt'',\lambda''). 
\]

\begin{theorem}
  The morphisms $(\Theta,\theta)$ and $(\Psi,\psi)$ define an
  equivalence of categories with conjugation
  \[
    (\catC,\conjt,\lambda)\equiv(\catC'',\conjt'',\lambda'').
  \]
  Moreover, the following diagrams commute (up to natural
  isomorphisms): 
  \begin{equation}
    \label{eq:diagKM}
    \begin{split}
    \xymatrix{\catC''\ar[dr]^{G'}
      &
      \\
      \catC\ar[r]^{F}\ar[u]^{\Psi}_{\wr}
      &\catC'}
    \qquad\qquad
    \xymatrix{\catC'\ar[r]^{F'}\ar[dr]_{G}
      &\catC''\ar[d]_{\wr}^{\Theta}
      \\
      &
      \catC
    }
    \end{split}
  \end{equation}
\end{theorem}

\begin{proof}
  For $C\in\catC$, the maps
$
\bigl(
\begin{smallmatrix}
\Idfunc_{C}\\
0
\end{smallmatrix}\bigr)\colon\,C\to C\oplus\conjt(C)
$
and
$
(\Idfunc_{C},0)\colon\,C\oplus\conjt(C)\to C
$
satisfy
\[
  (\Idfunc_{C},0)\circ
  \bigl(
\begin{smallmatrix}
\Idfunc_{C}\\
0
\end{smallmatrix}\bigr)=\Idfunc_{C}
\quad\text{and}\quad
\bigl(
\begin{smallmatrix}
\Idfunc_{C}\\
0
\end{smallmatrix}\bigr)
\circ (\Idfunc_{C},0)=
\begin{pmatrix}
\Idfunc_{C}&0\\
0&0
\end{pmatrix}=
{\textstyle\frac12}(\Idfunc_{C\oplus\conjt(C)}-
\Delta_{C}\circ\Lambda_{C}^{-1}).
\]
Therefore, we may identify $C$ with
$\Theta(C\oplus\conjt(C),\Lambda_{C},\Delta_{C})$ to 
obtain a natural isomorphism $\xi\colon\Idfunc_{\catC}\natiso
\Theta\Psi$. Explicitly, for $C\in\catC$,
\begin{equation}
  \label{eq:xidef}
  \xi_C\colon C\to \Theta\Psi(C) \quad\text{is given by}\quad
  \xi_C=s_{\Psi(C)}\circ\bigl(
\begin{smallmatrix}
\Idfunc_{C}\\
0
\end{smallmatrix}\bigr).
\end{equation}

Now, consider $C''=(C,f,g)\in\catC''$, and write simply
$X$ for $\Theta(C'')$. By \eqref{second.eq} we have
\[
t_{C''}=t_{C''} s_{C''} t_{C''}={\textstyle\frac12}(t_{C''}-g f^{-1}
t_{C''}) 
\quad\text{and}\quad
s_{C''}=s_{C''} t_{C''} s_{C''}={\textstyle\frac12}(s_{C''}-s_{C''} g
f^{-1}), 
\]
hence
\begin{equation}
\label{efgs1.eq}
g^{-1} t_{C''}= -f^{-1} t_{C''}
\qquad\text{and}\qquad
s_{C''} f=-s_{C''} g.
\end{equation}
Since $\conjt(f g^{-1})=-g^{-1} f$ by \eqref{fgantic.eq}, applying
$\conjt$ to \eqref{efgs1.eq} yields
\begin{equation}
\label{efgs2.eq}
f\conjt(t_{C''})=g\conjt(t_{C''})
\qquad\text{and}\qquad
\conjt(s_{C''}) g^{-1}=\conjt(s_{C''}) f^{-1}.
\end{equation}
\relax From \eqref{efgs1.eq} we get $\conjt(s_{C''}) g^{-1} t_{C''}=-
\conjt(s_{C''}) f^{-1}
t_{C''}$, and from \eqref{efgs2.eq},
$\conjt(s_{C''}) g^{-1} t_{C''}=\conjt(s_{C''}) f^{-1} 
t_{C''}$, hence $\conjt(s_{C''})
f^{-1}t_{C''}=0$ since $\frac12\in k$. Similarly, $s_{C''}
f\conjt(t_{C''})=0$. 
Computation then shows that the
maps
\[
(t_{C''},f\conjt(t_{C''}))\colon\,X\oplus\conjt(X)\to C
\qquad\text{and}\qquad
\bigl(
\begin{smallmatrix}
s_{C''}\\
\conjt(s_{C''}) f^{-1}
\end{smallmatrix}
\bigr)\colon\,C\to X\oplus\conjt(X)
\]
satisfy
\[
(t_{C''},f\conjt(t_{C''}))\circ\bigl(
\begin{smallmatrix}
s_{C''}\\
\conjt(s_{C''}) f^{-1}
\end{smallmatrix}
\bigr)=\Idfunc_{C}
\qquad\text{and}\qquad
\bigl(
\begin{smallmatrix}
s_{C''}\\
\conjt(s_{C''}) f^{-1}
\end{smallmatrix}
\bigr)\circ(t_{C''},f\conjt(t_{C''}))=\Idfunc_{X\oplus\conjt(X)},
\]
hence they define an isomorphism $C\simeq X\oplus\conjt(X)$ in
$\catC$. From~\eqref{efgs1.eq} and \eqref{efgs2.eq} it follows that
these maps actually define an isomorphism
$C''\simeq(X\oplus\conjt(X),\Lambda_{X},\Delta_{X})$ in 
$\catC''$, hence a natural isomorphism $\eta\colon
\Idfunc_{\catC''}\natiso 
\Psi\Theta$. Explicitly, for $C''\in\catC''$,
\begin{equation}
  \label{eq:etadef}
  \eta_{C''}\colon C''\to\Psi\Theta(C'') \quad\text{is given by}\quad
  \eta_{C''} =
  \bigl(
\begin{smallmatrix}
s_{C''}\\
\conjt(s_{C''}) f^{-1}
\end{smallmatrix}
\bigr).
\end{equation}

To complete the proof, we check the commutativity of the diagrams
in~\eqref{eq:diagKM}. For $C\in\catC$,
\[
  G'\Psi(C)=G'\bigl(C\oplus\conjt(C),\Lambda_C,\Delta_C\bigr) =
  (C\oplus\conjt(C),\Lambda_C) = F(C),
\]
hence $G'\Psi=F$ and the left triangle in~\eqref{eq:diagKM}
commutes.
Now, for $C'=(C,f)\in\catC'$,
\[
F'(C')=
\Bigl(
C\oplus\conjt(C),
\bigl(
\begin{smallmatrix}
f&0\\
0&-\conjt(f)
\end{smallmatrix}
\bigr),
\bigl(
\begin{smallmatrix}
0&-\lambda_C\\
\Idfunc_{\conjt(C)}&0
\end{smallmatrix}
\bigr)
\Bigr),
\]
and $\Theta F'(C')$ is a splitting of the idempotent $\frac12\bigl(
\begin{smallmatrix}
\Idfunc_C&-f\\
-f^{-1}&\Idfunc_{\conjt(C)}
\end{smallmatrix}
\bigr)$, equipped with morphisms in $\catC$
$\Theta F'(C')\xrightarrow{t_{F'(C')}}C\oplus\conjt(C)
\xrightarrow{s_{F'(C')}} \Theta F'(C')$
such that
\[
  s_{F'(C')}\circ t_{F'(C')} = \Idfunc_{\Theta F'(C')}
  \qquad\text{and}\qquad
  t_{F'(C')}\circ s_{F'(C')} = {\textstyle\frac12}
\begin{pmatrix}
\Idfunc_C&-f\\
-f^{-1}&\Idfunc_{\conjt(C)}
\end{pmatrix}.
\]
Observe that
$
  \bigl(\begin{smallmatrix}
\Idfunc_C\\-f^{-1}
\end{smallmatrix}
\bigr)\colon C\to C\oplus\conjt(C)
$
and
$
  \textstyle{\frac12}\bigl(
\Idfunc_C,-f
\bigr)\colon C\oplus\conjt(C)\to C
$
also satisfy
\[
  \textstyle{\frac12}\bigl(
\Idfunc_C,-f
\bigr)
\circ \bigl(\begin{smallmatrix}
\Idfunc_C\\-f^{-1}
\end{smallmatrix}
\bigr)=\Idfunc_C
\qquad\text{and}\qquad
\bigl(\begin{smallmatrix}
\Idfunc_C\\-f^{-1}
\end{smallmatrix}
\bigr)\circ \textstyle{\frac12}\bigl(
\Idfunc_C,-f
\bigr)={\textstyle\frac12}
\begin{pmatrix}
\Idfunc_C&-f\\
-f^{-1}&\Idfunc_{\conjt(C)}
\end{pmatrix}.
\]
Therefore, we may define a natural isomorphism $\zeta\colon G\natiso
\Theta F'$ as follows: for $C'=(C,f)\in\catC'$,
\begin{equation}
  \label{eq:zetadef}
  \zeta_{C'}\colon G(C')\to \Theta F'(C') \quad\text{is given by}\quad
  \zeta_{C'}=s_{F'(C')}\circ
  \bigl(\begin{smallmatrix}
\Idfunc_C\\-f^{-1}
\end{smallmatrix}
\bigr).
\end{equation}
Thus, the right-hand side diagram in~\eqref{eq:diagKM} commutes up to
a natural isomorphism.
\end{proof}

Since $\Psi\Theta\cong\Idfunc_{\catC''}$, the diagram on the left-hand
side of~\eqref{eq:diagKM} yields another commutative triangle (up to
a natural isomorphism):
\[
  \xymatrix{\catC''\ar[d]^{\wr}_{\Theta}\ar[dr]^{G'}&
    \\
    \catC\ar[r]^{F}&\catC'
  }
\]
The latter triangle and all its corresponding versions for $\catC'$,
$\catC''$, $\catC'''$, etc. can be pasted with the triangle on the
right-hand side of~\eqref{eq:diagKM} and all its corresponding
versions into the following infinite diagram:
\[
  \xymatrix{
    \catC'\ar[r]^{F'}\ar[dr]_{G}
    &
    \catC''\ar[r]^{F''}\ar[dr]^{G'}\ar[d]^{\Theta}
    &
    \catC'''\ar[r]^{F'''}\ar[dr]^{G''}\ar[d]^{\Theta'}
    &
    \catC''''\ar[r]^{F''''}\ar[dr]^{G'''}\ar[d]^{\Theta''}
    &
    \catC'''''\ar[r]^{F'''''}\ar[dr]^{G''''}\ar[d]^{\Theta'''}
    &
    \cdots
    \\
    &
    \catC\ar[r]^{F}
    &
    \catC'\ar[r]^{F'}\ar[dr]_{G}
    &
    \catC''\ar[r]^{F''}\ar[dr]^{G'}\ar[d]^{\Theta}
    &
    \catC'''\ar[r]^{F'''}\ar[dr]^{G''}\ar[d]^{\Theta'}
    &
    \cdots
    \\
    &
    &
    &
    \catC\ar[r]^{F}
    &
    \catC'\ar[r]^{F'}\ar[dr]_{G}
    &
    \cdots
    \\
    &&&&&
    \ddots
  }
\]
Retaining only the upper line and the left-most broken diagonal yields
a commutative diagram that demonstrates the periodicity of the
doubling construction:
\begin{equation}
  \label{eq:diagperiod}
  \begin{split}
    \xymatrix{
      \catC\ar[r]^{F} \ar@{=}[d]
      & \catC'\ar@{=}[d]\ar[r]^{F'}
      & \catC''\ar[d]_{\wr}^{\Theta} \ar[r]^{F''}
      & \catC''' \ar[d]_{\wr}^{\Theta'} \ar[r]^{F'''}
      & \catC'''' \ar[d]_{\wr}^{\Theta\Theta''} \ar[r]^{F''''}
      & \catC''''' \ar[d]_{\wr}^{\Theta'\Theta'''} \ar[r]^{F'''''}
      &\cdots
      \\
      \catC\ar[r]^{F}
      &
      \catC'\ar[r]^{G}
      &
      \catC\ar[r]^{F}
      &
      \catC'\ar[r]^{G}
      &
      \catC\ar[r]^{F}
      &
      \catC'\ar[r]^{G}
      &\cdots
      }
  \end{split}
\end{equation}

\subsection{Example: Modules}
\label{subsec:exconjug1}
To obtain a typical example of a category with conjugation and its
double, consider an 
arbitrary (associative) $k$-algebra $A$ and a pair $(\conjt,\llambda)$
consisting of a $k$-linear automorphism $\conjt\colon A\to A$ and a
unit $\llambda\in A^\times$ such that $\conjt(\llambda)=\llambda$ and
$\conjt^2(a)=\llambda a\llambda^{-1}$ for all $a\in A$. Define a
$k$-algebra 
\[
  A'=A\oplus Aj
\]
where the multiplication extends the multiplication in $A$ by
\[
  j^2=\llambda\qquad\text{and}\qquad ja=\conjt(a)j\quad\text{for $a\in
    A$.} 
\]
In Ranicki's terminology \cite[\S3]{R} the pair $(\conjt,\llambda)$ is
a \emph{structure} on $A$ and $A'$ is the
\emph{$(\conjt,\llambda)$-twisted quadratic extension} of $A$.

Let $\catM_A$ be the category of right
$A$-modules. Abusing notation, we write again $\conjt$ for the
endofunctor on $\catM_A$ defined as follows: for $M\in\catM_A$, we set
\[
\conjt(M)=\{{}^\conjt{x}\mid x\in M\}
\]
with the operations
\[
{}^\conjt{x}+{}^\conjt{y}={}^\conjt(x+y), \quad
{}^\conjt{x}\cdot a={}^\conjt(x\conjt(a))\qquad
\text{for $x$, $y\in M$ and $a\in A$}.
\]
For every morphism $\varphi\colon\,M\to N$, define
$\conjt(\varphi)\colon\,\conjt(M)\to\conjt(N)$ by
\[
\conjt(\varphi)({}^\conjt{x})={}^\conjt(\varphi(x))\qquad
\text{for $x\in M$}.
\]
Using the $\conjt$-fixed unit $\llambda\in A$ in the definition of
$A'$, we define a natural isomorphism
$
\lambda\colon\,\conjt\conjt\natiso\Idfunc_{\catM_{A}}
$
as follows:
\[
\lambda_{M}\colon\,\conjt\conjt(M)\to M,\quad
{}^{\conjt\conjt}x\mapsto x\llambda\qquad 
\text{for every $M\in\catM_A$ and every $x\in M$}.
\]
The pair $(\conjt,\lambda)$ is the conjugation on $\catM_{A}$ derived
from the structure $(\conjt,\llambda)$ on $A$.

Every object $(M,f)$ in the double category $\catM'_A$ can be regarded
as a right module over $A'$ by defining
\[
  x\cdot(a_1+a_2j) = xa_1+f({}^\conjt x)\conjt^{-1}(a_2)
  \qquad\text{for $x\in M$ and $a_1$, $a_2\in A$.}
\]
Under this identification, the morphisms in $\catM'_A$ are the
$A'$-linear maps.
Conversely, every right
$A'$-module $M'$ can be viewed as an object $(G(M'),f)\in\catM'_A$
where $G(M')$ is the $A$-module obtained from $M'$ by restriction of
scalars and
\[
  f\colon\conjt G(M')\to G(M'),\qquad {}^\conjt x\mapsto xj
  \qquad\text{for $x\in G(M')$.}
\]
Therefore, $\catM'_A$ is isomorphic to the category of right
$A'$-modules; henceforth we identify these 
categories: $\catM'_A=\catM_{A'}$. The endofunctor $\conjt'$ on
$\catM_{A'}$ is then derived from the automorphism
\[
  \conjt'\colon A'\to A',\qquad
  a_1+a_2j\mapsto \conjt(a_1)-\conjt(a_2)j \qquad\text{for $a_1$,
    $a_2\in A$,}
\]
just like the endofunctor $\conjt$ on $\catM_A$ is derived from the
automorphism $\conjt$  of $A$. The natural isomorphism
$\lambda'\colon\conjt'\conjt'\natiso\Idfunc_{\catM_{A'}}$ is given
by $\lambda'_{M'}=-\lambda_{G(M')}$ for $M'\in\catM_{A'}$, hence
\[
  \lambda'_{M'}({}^{\conjt'\conjt'}x)= -x\llambda\qquad\text{for $x\in
    M'$}. 
\]
Therefore, $\catM_A''$, the double category of $\catM'_A=\catM_{A'}$,
can be identified with the category $\catM_{A''}$ of right modules
over the $k$-algebra $A''=A'\oplus A'j'$, where
\[
  {j'}^2=-\llambda\qquad\text{and}\qquad j'a'=\conjt'(a')j'
  \qquad\text{for $a'\in A'$.}
\]
The conjugation $(\conjt',\lambda')$ on $\catM_{A'}$ is thus derived
from the structure $(\conjt',\llambda')$ on $A'$ with the unit
$\llambda'=-\llambda\in {A'}^\times$, and the $k$-algebra $A''$ is the
$(\conjt',\llambda')$-twisted quadratic extension of $A'$.

Under the identification $\catM'_A=\catM_{A'}$, the functor
$F\colon\catM_A\to\catM_{A'}$ is the scalar extension: for
$M\in\catM_A$, the object $(M\oplus\conjt(M),\Lambda_P)$ is identified
with $M\otimes_AA'$ under the map
\[
  (x,\,{}^\conjt y)\mapsto x\otimes1+y\otimes j \qquad\text{for $x$,
    $y\in M$.}
\]
Moreover, under the identification $\catM''_A=\catM_{A''}$, the
equivalence $\catM_A\equiv\catM_{A''}$ afforded 
by the functors $\Theta$ and $\Psi$ is a Morita equivalence. To see
this, note 
that $j^{-1}j'\in A''$ satisfies $(j^{-1}j')^2=1$, hence the elements
$
e_+=\textstyle{\frac12}(1+j^{-1}j')
$
and
$
  e_-=\textstyle{\frac12}(1-j^{-1}j')
$
are orthogonal idempotents in $A''$. For $M''\in\catM_{A''}$ and
$M\in \catM_A$ we may identify $\Theta(M'')=M''\otimes_{A''}A''e_-$
and 
$\Psi(M)=M\otimes_Ae_-A''$. We omit the proof, as this result is not
used subsequently.
\medbreak

Of interest for the sequel are two subcategories of $\catM_A$ that are
preserved under $\conjt$: let $\catP_A$ (resp.\ $\catF_A$) be the full
subcategory of finitely generated projective modules (resp.\ of
modules of finite length) in $\catM_A$. Since $A'$ is a free module of
finite rank over $A$, for every module $P'\in\catP_{A'}$ the
$A$-module $G(P')$ lies in $\catP_A$. Conversely, if $M'\in\catM_{A'}$
is such that $G(M')\in\catP_A$, then the isomorphism
$M'\oplus\conjt'(M')\simeq FG(M')=G(M')\otimes_AA'$ from
Proposition~\ref{prop:FG} shows that $M'$ is a direct summand of a
finitely generated projective $A'$-module, hence
$M'\in\catP_{A'}$. Therefore, 
the identification $\catM'_A=\catM_{A'}$ restricts to
$\catP'_A=\catP_{A'}$.

Similarly, $\catM'_A=\catM_{A'}$ yields $\catF'_{A}=\catF_{A'}$: to
see this, note that $G$ turns every repetition-free chain of
submodules in a module 
$M'\in\catM_{A'}$ into a repetition-free chain of submodules in
$G(M')$. Therefore, if 
$G(M')\in\catF_A$, then $M'\in\catF_{A'}$. Conversely, if
$M'\in\catF_{A'}$ then $G(M')\otimes_AA'\in\catF_{A'}$ by
Proposition~\ref{prop:FG}, hence $G(M')\in\catF_A$ since $F$ turns
every repetition-free chain of submodules in $G(M')$ into a
repetition-free chain of submodules in 
$G(M')\otimes_AA'$. 

\begin{remark}
  \label{rem:bimod}
  The example discussed in this subsection is not the most general
  conjugation on a category of modules: in a general conjugation
  $(\conjt, \lambda)$ on $\catM_A$ the functor $\conjt$ is an
  auto-equivalence of $\catM_A$, hence Morita theory shows that, up to
  isomorphism, $\conjt$ is given by
  \[
    \conjt(M)=M\otimes_AP \qquad\text{for some invertible
      $A$-$A$-bimodule $P$.}
  \]
  (See~\cite[Ch.~II, \S5]{Bass}.) The isomorphism $\lambda$ is given
  by
  \[
    \lambda_M\colon M\otimes_AP\otimes_AP \to M, \qquad
    m\otimes p\otimes q\mapsto m\,\beta(p\otimes q)
  \]
  for some isomorphism of bimodules $\beta\colon P\otimes_AP\to A$
  satisfying the associativity condition $\beta(p_1\otimes p_2)\otimes
  p_3 = p_1\beta(p_2\otimes p_3)$ for $p_1$, $p_2$, $p_3\in
  P$. According to~\cite[Prop.~5.2, p.~73]{Bass}, the example
  discussed above is the special case where $P$ is isomorphic to $A$
  as a left $A$-module. It then yields the automorphism of $A$ denoted
  (abusively) by $\conjt$ above because, up to isomorphism,
  $P=\{{}_1x_\conjt \mid x\in A\}$ where $a\cdot{}_1x_\conjt =
  {}_1ax_\conjt$ and ${}_1x_\conjt\cdot a={}_1x\conjt(a)_\conjt$ for
  $x$, $a\in A$. The map $\beta\colon P\otimes_AP\to A$ is then given
  by $\beta({}_1x_\conjt\otimes{}_1y_\conjt) = x\conjt(y)\lambda$
  where $\lambda\in A^\times$ is subject to $\conjt(\lambda)=\lambda$
  and $\conjt^2(a)=\lambda a\lambda^{-1}$ for all $a\in A$.

  When $A$ is commutative, the pairs $(P,\beta)$ that yield a
  conjugation on $\catM_A$ with the same action of $A$ on $P$ on the
  left and on the right are exactly the discriminant modules discussed
  in~\cite[Ch.~III, \S3]{Knus}. Since $2$ is invertible in $A$, these
  discriminant modules are in one-to-one correspondence with the
  separable quadratic $A$-algebras,
  see~\cite[Prop.~III(4.2.4)]{Knus}. Therefore, every 
  separable quadratic $A$-algebra arises as $A'$ for a suitable
  conjugation $(\conjt,\lambda)$ on $\catM_A$.
\end{remark}

\section{Conjugations and dualities}
\label{sec:dual}

Recall from~\cite[\S1.2.1]{Balmer} that a \emph{duality} on a
category $\catC$ is a pair $(D,\delta)$ consisting of a 
contravariant functor
$
  D\colon\,\catC\to\catC
$
and a natural isomorphism
$
  \delta\colon\,\Idfunc_{\catC}\natiso D D 
$
such that $(\Idnt_{D}\god\delta)\circ(\delta\god\Idnt_{D})=\Idnt_{D}$,
i.e.,\ for every object $C$ in $\catC$,
\[
\delta_{D(C)}=D(\delta_{C})^{-1}\colon\,D(C)\to
DDD(C).
\]
Note that for any duality $(D,\delta)$, the pair $(D,-\delta)$ also is
a duality.

A \emph{morphism} of categories with duality
$(\catC_{1},D_{1},\delta_{1})\to(\catC_{2},D_{2},\delta_{2})$ 
is a pair $(R,\theta)$ consisting of a functor $R\colon\,\catC_{1}\to\catC_{2}$
and a natural isomorphism $\theta\colon\,R D_{1}\natiso D_{2}
R$ 
such that the following diagram commutes:
\[
\xymatrix{
R
\ar@{=>}[r]^{\Idnt_{R}\god\delta_{1}}
\ar@{=>}[d]_{\delta_{2}\god\Idnt_{R}}
&
R{D_{1}}{D_{1}}
\ar@{=>}[d]^{\theta\god\Idnt_{D_{1}}}
\\
D_{2}D_{2} R
\ar@{=>}[r]_{\Idnt_{D_{2}}\god\theta}
&
D_{2} R D_{1}
}
\]

Throughout this section, we only consider $k$-linear categories
subject to the same conditions as in~\S\ref{sec:double}: $k$ is a
commutative ring in which $2$ is invertible, and every idempotent
morphism splits. All the functors we consider are $k$-linear. A notion
of \emph{similitude} of morphisms of 
categories with duality is defined as follows: morphisms
$\widetilde R=(R,\theta)$ and $\widetilde S=(S,\tau)$ from
$(\catC_1,D_1,\delta_1)$ to $(\catC_2,D_2,\delta_2)$ are said to be
\emph{similar} if there exists a natural isomorphism $\rho\colon
R\natiso S$ and a invertible scalar $\mu\in k^\times$ such that the
following diagram commutes:
\[
  \xymatrix{
    RD_1\ar@{=>}[d]_{\rho\god\Idnt_{D_1}} \ar@{=>}[r]^{\theta}
    & D_2R \ar@{=>}[rr]^{\mu\,\Idnt_{D_2R}}&& D_2R\\
    SD_1\ar@{=>}[rrr]^{\tau}&&&D_2S\ar@{=>}[u]_{\Idnt_{D_2}\god\rho}
  }
\]
i.e., for $C\in\catC_1$,
\[
  D_2(\rho_C)\circ\tau_C\circ\rho_{D_1(C)}=\mu\,\theta_C.
\]
The natural isomorphism $\rho$ is then said to be a \emph{similitude}
$\widetilde R \Rightarrow\widetilde S$ with \emph{multiplier}
$\mu$. Similitudes 
with multiplier~$1$ are \emph{natural isomorphisms} of morphisms of
categories with duality.

Diagram chase establishes the following multiplicativity property: for
any similitude $\rho_1\colon\widetilde R_1\Rightarrow \widetilde S_1$
of morphisms from $(\catC_1,D_1,\delta_1)$ to $(\catC_2,D_2,\delta_2)$
with multiplier $\mu_1$ and any similitude $\rho_2\colon\widetilde
R_2\Rightarrow \widetilde S_2$ of morphisms from
$(\catC_2,D_2,\delta_2)$ to $(\catC_3,D_3,\delta_3)$ with multiplier
$\mu_2$, the product $\rho_2\god\rho_1$ is a similitude $\widetilde
R_2\widetilde R_1\Rightarrow \widetilde S_2\widetilde S_1$ with
multiplier $\mu_1\mu_2$.

\subsection{Commuting conjugations and dualities}

A \emph{commutation} between a conjugation $(\conjt,\lambda)$ and a
duality 
$(D,\delta)$ on a category $\catC$ is a natural isomorphism
\[
  \alpha\colon\,\conjt D\natiso D\conjt
\]
such that the following diagrams commute:
\begin{equation}
  \label{conjdual.cd}
  \begin{split}
  \xymatrix{
  \conjt\ar@{=>}[r]^{\Idnt_{\conjt}\god\delta}
  \ar@{=>}[d]_{\delta\god\Idnt_{\conjt}}&
  \conjt{D}{D}
  \ar@{=>}[d]^{\alpha\god\Idnt_{D}}
  \\
  {D}{D}\conjt
  \ar@{=>}[r]_{\Idnt_{D}\god\alpha}&
  D\conjt D
  }
  \qquad\qquad
  \xymatrix{
  D
  \ar@{=>}[r]^{\Idnt_{D}\god\lambda}
  \ar@{=>}[d]_{\lambda^{-1}\god\Idnt_{D}}
  &
  D\conjt\conjt
  \ar@{=>}[d]^{\alpha^{-1}\god\Idnt_{\conjt}}
  \\
  \conjt\conjt D
  \ar@{=>}[r]_{\Idnt_{\conjt}\god\alpha}
  &
  \conjt{D}\conjt
}
\end{split}
\end{equation}

Note that a commutation between $(\conjt,\lambda)$ and $(D,\delta)$
also defines a commutation between $(\conjt,\lambda)$ and
$(D,-\delta)$.
Commuting conjugations and dualities on a category $\catC$ yield new
dualities on $\catC$ and on the double category $\catC'$ as follows:

\begin{proposition}
  \label{prop:altdual}
  Given a conjugation $(\conjt,\lambda)$ and a duality
  $(D,\delta)$ with commutation $\alpha$ on a category $\catC$,
  dualities $(\conjt D,\delta_\alpha)$ and
  $(D\conjt,\delta_{\alpha^{-1}})$ are defined by
  \begin{align*}
    \delta_\alpha&=(\Idnt_\conjt\god\alpha\god\Idnt_D)\circ
    (\lambda^{-1}\god\Idnt_{DD})\circ\delta\colon
    \Idfunc_\catC\natiso\conjt D\conjt D
    \\
    \delta_{\alpha^{-1}}&=(\Idnt_D\god\alpha\god\Idnt_\conjt) \circ
        (\delta\god\Idnt_{\conjt\conjt}) \circ \lambda^{-1}\colon
                          \Idfunc_\catC\natiso D\conjt D\conjt.
  \end{align*}
  Moreover, $(\Idfunc_\catC,\alpha)$ is an isomorphism of categories
  with duality $(\catC,\conjt D,\delta_\alpha) \isom (\catC,D\conjt,
  \delta_{\alpha^{-1}})$. 
\end{proposition}

\begin{proposition}
  \label{prop:doubledual}
  Given a conjugation $(\conjt,\lambda)$ and a duality
  $(D,\delta)$ with commutation $\alpha$ on a category $\catC$, a
  duality $(D',\delta')$ on the double category $\catC'$ is defined as
  follows: for every
$(C,f)\in\catC'$ and every morphism $\varphi$ in $\catC'$,
\[
D'(C,f)=(D(C),-D(f)^{-1}\circ\alpha_{C}),\qquad
D'(\varphi)=D(\varphi),\qquad
\delta'_{(C,f)}=\delta_{C}.
\]
Moreover, setting $\alpha'_{(C,f)}=-\alpha_{C}$ for $(C,f)\in\catC'$
defines a natural isomorphism 
\[
\alpha'\colon\,\conjt' D'\natiso D'\conjt',
 \]
 which is a commutation between $(\conjt',\lambda')$ and
 $(D',\delta')$. 
\end{proposition}

We omit the proofs of Propositions~\ref{prop:altdual} and
\ref{prop:doubledual}, which consist in lengthy but straightforward
computations. 
\medbreak

Recall from \S~\ref{subsec:def} the functor
$F\colon\,\catC\to\catC'$. 
Define natural isomorphisms
\[
  \widehat\phi\colon\,FD\natiso D'F\qquad\text{and}\qquad
  \widecheck\phi\colon\,F\conjt D\natiso \conjt' D'F
\]
as follows: for $C\in\catC$,
\[
\widehat\phi_{C}=
\begin{pmatrix}
\Idfunc_{D(C)}&0\\
0&-\alpha_{C}
\end{pmatrix}
\qquad\text{and}\qquad
\widecheck\phi_{C}=
\begin{pmatrix}
\Idfunc_{\conjt D(C)}&0\\
0&\conjt(\alpha_{C})
\end{pmatrix}.
\]
Computation shows that the pairs $\widehat F=(F,\widehat\phi)$ and
$\widecheck F=(F,\widecheck\phi)$ are
morphisms of categories with duality
\[
\widehat F\colon\,(\catC,D,\pm\delta)\to
(\catC',D',\pm\delta')
\qquad\text{and}\qquad
\widecheck F\colon\,(\catC,\conjt D,\pm\delta_\alpha) \to
(\catC',\conjt'D',\pm\delta'_{\alpha'}).
\]

Iterating the $\widehat F$ construction, we obtain an infinite
sequence of 
morphisms of categories with duality:
\begin{equation}
\label{cat.seq}
(\catC,D,\delta)\xrightarrow{\widehat F}
(\catC',D',\delta')\xrightarrow{\widehat F'}
(\catC'',D'',\delta'')\xrightarrow{\widehat F''}\cdots
\end{equation}

Morphisms in the reverse direction are provided by the
forgetful functor $G\colon\catC'\to\catC$, but they involve different
dualities: it is easily checked that 
natural isomorphisms $\widehat\gamma\colon GD'\natiso \conjt DG$ and
$\widecheck\gamma\colon G\conjt'D'\natiso DG$ are defined by
\[
  \widehat\gamma_{(C,f)}= \alpha_C^{-1}\circ D(f)
  \qquad\text{and}\qquad
  \widecheck\gamma_{(C,f)}= D(f^{-1})\circ\alpha_C
  \qquad\text{for $(C,f)\in\catC'$}.
\]
The pairs $\widehat G=(G,\widehat\gamma)$ and $\widecheck
G=(G,\widecheck\gamma)$ are morphisms of categories with duality:
\[
  \widehat G\colon(\catC',D',\pm\delta')\to
  (\catC,\conjt D,\mp \delta_\alpha)
  \qquad\text{and}\qquad
  \widecheck G\colon (\catC',\conjt'
  D',\pm\delta'_{\alpha'}) \to (\catC,D, \pm\delta).
\]
This change of dualities (and signs) explains why the
sequence~\eqref{cat.seq} 
turns out to be $8$-periodic, in contrast with the $2$-periodicity of
the sequence~\eqref{eq:diagperiod}.

\subsection{Periodicity}
\label{subsec:periodual}

Continuing with the same notation as in the preceding section, we turn
the equivalence of categories $\catC\equiv\catC''$ afforded by the
functors $\Theta$ and $\Psi$ of \S\ref{subsec:periodconj} into
equivalences of categories with duality.

Natural isomorphisms $\widehat\psi\colon\Psi D\natiso
  \conjt''D''\Psi$ and $\widecheck\psi\colon \Psi\conjt D\natiso
  D''\Psi$ are defined by the following formulas, for $C\in\catC$:
  \[
    \widehat\psi_C=
    \begin{pmatrix}
      0&\Idfunc_{\conjt D(C)}\\
      \conjt(\alpha_C)\circ\lambda_{D(C)}^{-1}&0
    \end{pmatrix}
    \qquad\text{and}\qquad
    \widecheck\psi_C=
    \begin{pmatrix}
      0&\lambda_{D(C)}\\
      -\alpha_C&0
    \end{pmatrix}
    .
  \]
  The pairs $\widehat\Psi=(\Psi,\widehat\psi)$ and
  $\widecheck\Psi=(\Psi,\widecheck\psi)$ are morphisms of categories
  with duality
  \[
    \widehat\Psi\colon(\catC,D,\pm\delta) \to (\catC'',\conjt''D'',
    \pm\delta_{\alpha''}'')
    \qquad\text{and}\qquad
    \widecheck\Psi\colon(\catC,\conjt D,\pm\delta_\alpha) \to
    (\catC'',D'',\mp\delta'').
  \]
Similarly, natural isomorphisms $\widehat\theta\colon \Theta D''
\natiso \conjt D\Theta$ and $\widecheck\theta\colon
\Theta\conjt''D''\natiso D\Theta$ are defined as follows: for
$C''=(C,f,g)\in\catC''$,
\[
  \widehat\theta_{C''}=-\conjt D(t_{C''}) \circ\alpha_C^{-1}\circ
  D(f)\circ t_{D''(C'')}\qquad\text{and}\qquad
  \widecheck\theta_{C''}=D(t_{C''})\circ D(f^{-1})\circ \alpha_C\circ
  t_{\conjt''D''(C'')}.
\]
The pairs $\widehat\Theta=(\Theta,\widehat\theta)$ and
$\widecheck\Theta=(\Theta,\widecheck\theta)$ are morphisms of
categories with duality
\[
  \widehat\Theta\colon(\catC'',D'',\pm\delta'')\to (\catC, \conjt D,
  \mp\delta_\alpha)
  \qquad\text{and}\qquad
  \widecheck\Theta\colon(\catC'',\conjt''D'',\pm\delta''_{\alpha''})
  \to (\catC,D,\pm\delta).
\]

\begin{theorem}
  \label{thm:PsiTheta}
  The morphisms $\widehat\Psi$, $\widecheck\Psi$ and
  $\widehat\Theta$, $\widecheck\Theta$ define equivalences of
  categories with duality
  \[
    \xymatrix{(\catC,D,\pm\delta)
  \ar@<0.5ex>[r]^-{\widehat\Psi}
  &
\ar@<0.5ex>[l]^-{\widecheck\Theta}
(\catC'',\conjt''D'',\pm\delta_{\alpha''}'')}
\qquad\text{and}\qquad
\xymatrix{(\catC,\conjt D,\pm\delta_\alpha)
  \ar@<0.5ex>[r]^-{\widecheck\Psi}
  &
\ar@<0.5ex>[l]^-{\widehat\Theta}
(\catC'',D'',\mp\delta'')}.
\]
Moreover, the following diagrams commute:
\begin{equation}
  \begin{split}
    \label{eq:diagFG'Psi}
  \xymatrix{
    (\catC'',\conjt''D'',\pm\delta_{\alpha''}'')
    \ar[dr]^{\widecheck G'}
    &
    \\
    (\catC,D,\pm\delta)\ar[r]_-{\widehat F}\ar[u]^{\widehat\Psi}_{\wr}&
    (\catC',D',\pm\delta')
  }
  \quad\quad
  \xymatrix{
    (\catC'',D'',\mp\delta'')
    \ar[dr]^{\widehat G'}
    &
    \\
    (\catC,\conjt D,\pm\delta_\alpha)\ar[r]_-{\widecheck F}
    \ar[u]^{\widecheck\Psi}_{\wr}&
    (\catC',\conjt'D',\pm\delta_{\alpha'}')
  }
  \end{split}
\end{equation}
and the following diagrams commute up to similitude:
\begin{equation}
  \begin{split}
    \label{eq:diagGF'Theta}
  \xymatrix{
    (\catC',D',\pm\delta')\ar[r]^-{\widehat F'}
    \ar[dr]_{\widehat G}
    &
    (\catC'',D'',\pm\delta'')
    \ar[d]^{\widehat\Theta}_{\wr}
    \\
    &
    (\catC,\conjt D,\mp\delta_\alpha)}
  \quad\quad
  \xymatrix{
    (\catC',\conjt'D',\pm\delta'_{\alpha'})\ar[r]^-{\widecheck F'}
    \ar[dr]_{\widecheck G}
    &
    (\catC'',\conjt''D'',\pm\delta''_{\alpha''})
    \ar[d]^{\widecheck\Theta}_{\wr}
    \\
    &
    (\catC,D,\mp\delta)}
  \end{split}
\end{equation}
\end{theorem}

\begin{proof}
  Recall from~\eqref{eq:xidef} the natural isomorphism $\xi\colon
  \Idfunc_{\catC}\natiso\Theta\Psi$; computation shows that
  \[
    (\Idnt_D\god\xi)\circ (\widecheck\theta\god\Idnt_\Psi) \circ
    (\Idnt_\Theta\god\widehat\psi) \circ (\xi\god\Idnt_D) = \Idnt_D
    \text{ and }
    (\Idnt_{\conjt D}\god\xi)\circ(\widehat\theta\god\Idnt_{\Psi})
    \circ (\Idnt_{\Theta}\god\widecheck\psi) \circ
    (\xi\god\Idnt_{\conjt D}) = \Idnt_{\conjt D}.
  \]
  Therefore, $\xi$ defines natural isomorphisms $(\Idfunc_{\catC},
  \Idnt_D) \natiso \widecheck\Theta\widehat\Psi$ and
  $(\Idfunc_{\catC}, \Idnt_{\conjt D})\natiso
  \widehat\Theta\widecheck\Psi$.

  Similarly, using the natural isomorphism
  $\eta\colon\Idfunc_{\catC''} \natiso \Psi\Theta$
  of~\eqref{eq:etadef}, one checks that
  \[
    (\Idnt_{\conjt''D''}\god\eta)\circ(\widehat\psi\god\Idnt_\Theta)
    \circ
    (\Idnt_\Psi\god\widecheck\theta)\circ(\eta\god\Idnt_{\conjt''D''})
    = \Idnt_{\conjt''D''}
    \text{ and }
    (\Idnt_{D''}\god\eta) \circ (\widecheck\psi\god\Idnt_\Theta) \circ
    (\Idnt_\Psi\god\widehat\theta)\circ(\eta\god\Idnt_{D''}) =
    \Idnt_{D''},
  \]
  hence $\eta$ defines natural isomorphisms $(\Idfunc_{\catC''},
  \Idnt_{\conjt''D''}) \natiso \widehat\Psi\widecheck\Theta$ and
  $(\Idfunc_{\catC''},\Idnt_{D''}) \natiso \widecheck\Psi
  \widehat\Theta$.

  Since $F=G'\Psi$, the commutativity of the
  diagrams~\eqref{eq:diagFG'Psi} is easily checked. Finally, to deal
  with the diagrams~\eqref{eq:diagGF'Theta}, we use the natural
  isomorphism $\zeta\colon G\natiso\Theta F'$ of~\eqref{eq:zetadef}
  and compute:
  \[
    (\Idnt_{\conjt D}\god\zeta)\circ(\widehat\theta\god\Idnt_{F'})
    \circ(\Idnt_\Theta\god\widehat\phi')\circ (\zeta\god\Idnt_{D'}) =
    -2\widehat\gamma
    \text{ and }
    (\Idnt_D\god\zeta)\circ(\widecheck\theta\god\Idnt_{F'})
    \circ(\Idnt_\Theta\god\widecheck\phi')\circ
    (\zeta\god\Idnt_{\conjt'D'})
    = 2\widecheck\gamma.
  \]
  Therefore, $\zeta\colon \widehat G\Rightarrow \widehat\Theta\widehat
  F'$ is a similitude with multiplier $-2$ and $\zeta\colon \widecheck
  G\Rightarrow \widecheck\Theta \widecheck F'$ is a similitude with
  multiplier $2$.
\end{proof}

Since
$\widehat\Psi\widecheck\Theta\cong
(\Idfunc_{\catC''},\Idnt_{\conjt''D''})$ 
and
$\widecheck\Psi\widehat\Theta\cong(\Idfunc_{\catC''},\Idnt_{D''})$,
the diagrams~\eqref{eq:diagFG'Psi} yield diagrams that commute up to
natural isomorphisms:
\[
  \xymatrix{
    (\catC'',\conjt''D'',\pm\delta_{\alpha''}'')
    \ar[dr]^{\widecheck G'}\ar[d]_{\widecheck\Theta}^{\wr}
    &
    \\
    (\catC,D,\pm\delta)\ar[r]_-{\widehat F}&
    (\catC',D',\pm\delta')
  }
  \quad\quad
  \xymatrix{
    (\catC'',D'',\mp\delta'')
    \ar[dr]^{\widehat G'}\ar[d]_{\widehat\Theta}^{\wr}
    &
    \\
    (\catC,\conjt D,\pm\delta_\alpha)\ar[r]_-{\widecheck F}
    &
    (\catC',\conjt'D',\pm\delta_{\alpha'}')
  }
\]
Pasting these diagrams with the diagrams~\eqref{eq:diagGF'Theta} for
the various doubles yields an infinite diagram that commutes up to
similitudes: 
\[
  \xymatrix{
    (\catC,D,\delta)\ar[r]^{\widehat F}
    &
    (\catC',D',\delta')\ar[r]^{\widehat F'}
    \ar[dr]_{\widehat G}
    &
    (\catC'',D'',\delta'')\ar[r]^{\widehat F''}
    \ar[d]_{\widehat\Theta}\ar[dr]^{\widehat G'}
    &
    (\catC''',D''',\delta''')\ar[r]^{\widehat F'''}
    \ar[d]_{\widehat\Theta'}\ar[dr]^{\widehat G''}
    &
    (\catC'''',D'''',\delta'''')\cdots
    \ar[d]_{\widehat\Theta''}
    \\
    &&
    (\catC,\conjt D,-\delta_\alpha)\ar[r]^{\widecheck F}
    &
    (\catC',\conjt'D',-\delta'_{\alpha'})\ar[r]^{\widecheck F'}
    \ar[dr]_{\widecheck G}
    &
    (\catC'',\conjt''D'',\delta''_{\alpha''})\cdots
    \ar[d]_{\widecheck\Theta}
    \\
    &&&&
    (\catC,D,-\delta)\cdots
  }
\]
Retaining only the upper sequence and the left-most broken diagonal,
we obtain a diagram that commutes up to similitudes:
\begin{equation}
  \label{eq:commdiagcatdual}
  \xymatrix{
    (\catC,D,\delta)\ar[r]^-{\widehat F}\ar@{=}[d]
    &
    (\catC',D',\delta')\ar[r]^-{\widehat F'}\ar@{=}[d]
    &
    (\catC'',D'',\delta'')\ar[d]_{\widehat \Theta}
    \ar[r]^-{\widehat F''}
    &
    (\catC''',D''',\delta''')\ar[d]_{\widehat \Theta'}
    \ar[r]^-{\widehat F'''}
    &
    (\catC'''',D'''',\delta'''')\cdots
    \ar[d]_{\widecheck\Theta\widehat \Theta''}
    \\
    (\catC,D,\delta)\ar[r]^-{\widehat F}
    &
    (\catC',D',\delta')\ar[r]^-{\widehat G}
    &
    (\catC,\conjt D,-\delta_\alpha)\ar[r]^-{\widecheck F}
    &
    (\catC',\conjt'D',-\delta'_{\alpha'})
    \ar[r]^-{\widecheck G}
    &
    (\catC,D,-\delta)\cdots
  }
\end{equation}
The lower sequence is $8$-periodic and folds into an octagon:
\begin{equation}
  \label{eq:octacatdual}
  \begin{split}
  \xymatrix{
    (\catC,D,\delta)\ar[r]^-{\widehat F}&
    (\catC',D',\delta')\ar[r]^-{\widehat G}&
    (\catC,\conjt D,-\delta_\alpha)\ar[d]^{\widecheck F}
    \\
    (\catC',\conjt'D',\delta'_{\alpha'})\ar[u]^{\widecheck G}
    &
    &
    (\catC',\conjt'D',-\delta'_{\alpha'})
    \ar[d]^{\widecheck G}
    \\
    (\catC,\conjt D,\delta_\alpha)\ar[u]^{\widecheck F}
    &
    (\catC',D',-\delta')\ar[l]_-{\widehat G}
    &
    (\catC,D,-\delta)\ar[l]_-{\widehat F}
  }
  \end{split}
\end{equation}

\subsection{Example: Projective modules}
\label{subsec:exdual}

The typical example of duality arises in the context of algebras with
involution or, more generally, with \emph{antistructure} in
C.T.C.~Wall's terminology~\cite{WallII}, \cite[\S1]{R}. Let $A$ be an
  associative $k$-algebra and $(\sigma,\varepsilon)$ a pair consisting
  of a $k$-linear antiautomorphism $\sigma\colon A\to A$ and a unit
  $\varepsilon\in A^\times$ such that
  $\sigma(\varepsilon)=\varepsilon^{-1}$ and $\sigma^2(a)=\varepsilon
  a \varepsilon^{-1}$ for all $a\in A$.
On the category~$\catP_A$ of finitely generated projective right
$A$-modules, a contravariant endofunctor $D$ is defined by
\[
  D(P)=\Hom_A(P,A) \qquad\text{for $P\in\catP_A$},
\]
using $\sigma$ to twist the natural left $A$-module structure on
$\Hom_A(P,A)$ into a right $A$-module structure. Thus, letting
$\langle\text{\textvisiblespace}, \text{\textvisiblespace}\rangle$
denote the 
canonical pairing between a module and its dual, we have
\[
  \langle\pi a,x\rangle = \sigma(a)\langle\pi,x\rangle \qquad\text{for
    $a\in A$, $\pi\in D(P)$ and $x\in P\in\catP_A$.}
\]
A natural isomorphism $\delta\colon\Idfunc_{\catP_A} \natiso DD$ is
defined by
\begin{equation*}
  \label{eq:delta}
  \langle\delta_P(x),\pi\rangle =
  \sigma(\langle\pi,x\rangle)\,\varepsilon \qquad\text{for $\pi\in
    D(P)$ and $x\in P\in\catP_A$.}
\end{equation*}
The pair $(D,\delta)$ is the duality on $\catP_A$ derived from the
antistructure $(\sigma,\varepsilon)$.
\medbreak

Now, recall the conjugation $(\conjt,\lambda)$ on $\catP_A$ from
Example~\ref{subsec:exconjug1}, derived from a structure
$(\conjt,\llambda)$ on $A$. Assume there exists a unit $u\in A^\times$
such that\footnote{The conditions on $u$ are necessary and sufficient
  to guarantee that \eqref{eq:sigma'} defines an antiautomorphism
  $\sigma'$ on $A'$ such that $\sigma'(j)j\in A^\times$ and
  ${\sigma'}^2(j)=\varepsilon j \varepsilon^{-1}$.}
  \begin{enumerate}
  \item[(i)]
    $\conjt\sigma(a)\,\conjt(u)=\conjt(u)\,\sigma\conjt(a)$ for all
    $a\in A$,
  \item[(ii)]
    $\conjt\sigma(\llambda)\llambda=\conjt(u)u$, and
  \item[(iii)]
    $\varepsilon\,\conjt(\varepsilon)^{-1} =
    \sigma\conjt(u)\,\conjt(u)^{-1}$. 
  \end{enumerate}
Then a commutation $\alpha\colon \conjt D\natiso D\conjt$ between
$(\conjt,\lambda)$ and $(D,\delta)$ is defined by
\begin{equation}
  \label{eq:defalphaP}
  \langle\alpha_P({}^\conjt\pi),{}^\conjt x\rangle = u\;
  \conjt^{-1}(\langle\pi,x\rangle)
  \qquad\text{for $\pi\in D(P)$ and $x\in P\in\catP_A$}.
\end{equation}

In order to describe the duality $(\conjt D,\delta_\alpha)$ defined in
Proposition~\ref{prop:altdual} in the same terms as the duality
$(D,\delta)$, we introduce the following notation:
\[
  \varepsilon_\alpha=
  \sigma(u\llambda^{-1})\varepsilon\quad\text{and}\quad 
  \{{}^\conjt\pi,x\}=\langle\pi,x\rangle \quad\text{for $\pi\in D(P)$
    and $x\in P\in\catP_A$.}
\]
Then $(\sigma\conjt,\varepsilon_\alpha)$ is an antistructure on $A$, and
$\{({}^\conjt\pi)a,x\}=\sigma\conjt(a)\,\{{}^\conjt\pi,x\}$ for
$a\in A$, ${}^\conjt\pi\in \conjt D(P)$ and $x\in P\in\catP_A$. Moreover,
\[
  \{(\delta_\alpha)_P(x),{}^\conjt\pi\} =
  \sigma\conjt(\{{}^\conjt\pi,
  x\})\,\varepsilon_\alpha
  \qquad\text{for ${}^\conjt\pi\in \conjt D(P)$ and $x\in
    P\in\catP_A$.}
\]
Therefore, the duality $(\conjt D,\delta_\alpha)$ is the duality on
$\catP_A$ derived from the
antistructure~$(\sigma\conjt,\varepsilon_\alpha)$ on $A$.
\medbreak

Next, we describe the duality $(D',\delta')$ on the double category
$\catP'_A$. Recall from Example~\ref{subsec:exconjug1} that $\catP'_A$
is identified with the category $\catP_{A'}$ of finitely generated
projective right modules over the algebra $A'=A\oplus Aj$, where
$j^2=\llambda$ and $ja=\conjt(a)j$ for $a\in A$: each pair
$P'=(P,f)\in\catP'_A$ is viewed as the $A'$-module obtained by
extending the $A$-module structure on $P$ by the rule $xj=f({}^\conjt
x)$ for $x\in P$. By definition, the $A'$-module structure on
$D'(P')=(D(P),-D(f)^{-1}\circ\alpha_P)$ is given by $\pi
j=-D(f)^{-1}\circ\alpha_P({}^\conjt\pi)$ for $\pi\in D(P)$, which
means that
\begin{equation}
  \label{eq:pij}
  \langle\pi j,x\rangle =
  -\langle\alpha_P({}^\conjt\pi),f^{-1}(x)\rangle =
  -u\;\conjt^{-1}(\langle\pi,xj^{-1}\rangle). 
\end{equation}
Extend the antiautomorphism $\sigma$ on $A$ to an antiautomorphism
$\sigma'$ on $A'$ such that ${\sigma'}^2(a')=\varepsilon
a'\varepsilon^{-1}$ for all $a'\in A'$ by
\begin{equation}
  \label{eq:sigma'}
  \sigma'(a_1+a_2j) = \sigma(a_1)-uj^{-1}\sigma(a_2)
  \qquad\text{for $a_1$, $a_2\in A$,}
\end{equation}
and let
\begin{equation}
  \label{eq:pairing'}
  \langle\pi,x\rangle' = \langle\pi,x\rangle+
  \langle\pi,xj^{-1}\rangle j\in A' \qquad\text{for $\pi\in D(P)$ and
    $x\in P$.}
\end{equation}
Then it follows from~\eqref{eq:pij} that $\langle\pi a',x\rangle' =
\sigma'(a')\langle\pi,x\rangle'$ for $a'\in A'$, $\pi\in D(P)$ and
$x\in P$, and
$\langle\text{\textvisiblespace},\text{\textvisiblespace}\rangle'$ is
an $A'$-sesquilinear pairing $D'(P')\times P'\to A'$, which identifies
$D'(P')$ with $\Hom_{A'}(P',A')$. 
With this notation, the definition $\delta'_{P'} = \delta_{G(P')}$ and
the conditions on $u$ imply
\[
  \langle\delta'_{P'}(x),\pi\rangle' = 
  \sigma'(\langle\pi, x\rangle')\,\varepsilon \qquad\text{for $\pi\in
    D'(P')$ and $x\in P'$.}
\]
Therefore, the duality $(D',\delta')$ on $\catP_{A'}$ is
derived from the antistructure $(\sigma',\varepsilon)$ on $A'$.

Defining $\{{}^{\conjt'}\pi,x\}' = \langle \pi, x\rangle'$ for $\pi\in
D'(P')$ and $x\in P'\in\catP_{A'}$, we also have
\[
  \{\delta'_{\alpha'}(x),{}^{\conjt'}\pi\}' =
  \sigma'\conjt'(\{{}^{\conjt'}\pi,x\}')\,\varepsilon_\alpha
  \quad\text{for ${}^{\conjt'}\pi\in\conjt'D'(P')$ and $x\in P'$,}
\]
because $\alpha'\colon\conjt'D'\natiso D'\conjt'$ and $\lambda'\colon
\conjt'\conjt'\natiso \Idfunc_{\catP_{A'}}$ are defined by
$\alpha'_{P'} = -\alpha_{G(P')}$ and $\lambda'_{P'} =
-\lambda_{G(P')}$. Thus, the duality $(\conjt'D',\delta'_{\alpha'})$
is derived from the antistructure $(\sigma'\conjt',\varepsilon)$ on
$A'$. 

\section{Symmetric spaces}
\label{sec:symsp}

As in the previous sections, $\catC$ denotes a pseudo-abelian
$k$-linear category for 
an arbitrary commutative ring $k$ in which $2$ is
invertible. Throughout the section, $(D,\delta)$ is a duality on
$\catC$.

\subsection{Definitions}
\label{subsec:defnsymsp}

A \emph{symmetric space} in the category with duality $(\catC,
D,\delta)$ is a pair $(C,h)$ consisting of an object $C\in\catC$ and
an isomorphism $h\colon C\to D(C)$ in $\catC$ such that
$D(h)\circ\delta_C=h$. For $\mu\in k^\times$, a \emph{similitude with
  multiplier $\mu$} from a symmetric space $(C_1,h_1)$ to a symmetric
space $(C_2,h_2)$ is an isomorphism $\varphi\colon
C_1\xrightarrow{\sim} C_2$ in $\catC$ such that the following diagram 
commutes:
\[
  \xymatrix{C_1\ar[rrr]^{\varphi} \ar[d]_{h_1} &&& C_2\ar[d]^{h_2}\\
    D(C_1)\ar[rr]^{\mu\Idfunc_{D(C_1)}}&&D(C_1)&
      D(C_2)\ar[l]_{D(\varphi)}
    }
  \]
Similitudes with multiplier~$1$ are \emph{isometries}.

For $L\in\catC$, let $H(L)$ denote the symmetric space
\[
  H(L)=\Bigl(L\oplus D(L),\;
  \bigl(\begin{smallmatrix}0&\Idfunc_{D(L)}\\ \delta_L&0
  \end{smallmatrix}\bigr)
  \Bigr).
\]
By definition, a symmetric space $(C,h)$ is \emph{hyperbolic} if
it is isometric to $H(L)$ for some $L\in\catC$. Note that for every
$\mu\in k^\times$ the morphism $\bigl(
\begin{smallmatrix}
  \mu\Idfunc_L&0\\
  0&\Idfunc_{D(L)}
\end{smallmatrix}
\bigr)$ is a similitude $H(L)\to H(L)$ with multiplier $\mu$. Therefore,
hyperbolicity is preserved under similitude.
      
\subsection{Hermitian categories}
\label{subsec:catsymsp}

Let $\catS(\catC, D, \delta)$ denote the category whose objects are
symmetric spaces in $(\catC,D,\delta)$ and morphisms are
isometries. We refer to $\catS(\catC,D,\delta)$ as the \emph{hermitian
  category} of the category with duality $(\catC,D,\delta)$.
Every morphism of categories with duality $\widetilde R=
(R,\theta)\colon (\catC_1,D_1,\delta_1)\to (\catC_2,D_2,\delta_2)$
induces a functor
\[
  \catS(\widetilde R)\colon \catS(\catC_1,D_1,\delta_1) \to
  \catS(\catC_2,D_2,\delta_2)
\]
mapping $(C,h)\in\catS(\catC_1,D_1,\delta_1)$ to
\[
  \catS(\widetilde R)(C,h) = \bigl(R(C), \theta_C\circ R(h)\bigr),
\]
see \cite[II(2.6)]{Knus} or \cite[Def.~11]{Balmer}. It readily follows
from the definitions that $\catS(\widetilde R)$ maps hyperbolic
symmetric spaces to hyperbolic symmetric spaces, because for every
$L\in\catC_1$ the morphism $\bigl(
\begin{smallmatrix}
  \Idfunc_{R(L)}&0\\ 0&\theta_L
\end{smallmatrix}
\bigr)\colon R(L)\oplus RD_1(L)\to R(L)\oplus D_2R(L)$ is an isometry
from $\catS(\widetilde R)\bigl(H(L)\bigr)$ to $HR(L)$. 

If $\rho\colon\widetilde R \natiso
\widetilde S$ is a similitude with multiplier $\mu\in k^\times$
between morphisms of categories with duality $\widetilde R$,
$\widetilde S\colon (\catC_1,D_1,\delta_1)\to (\catC_2,D_2,\delta_2)$,
then for every symmetric space $(C,h)\in\catS(\catC_1,D_1,h_1)$ the
isomorphism $\rho_C\colon R(C)\to S(C)$ in $\catC_2$ defines a
similitude with multiplier $\mu$
\[
  \catS(\rho_C)\colon \catS(\widetilde R)(C,h) \to \catS(\widetilde
  S)(C,h).
\]
Letting $\catS(\rho)_{(C,h)} = \catS(\rho_C)$, we also call
$\catS(\rho)\colon \catS(\widetilde R)\natiso \catS(\widetilde S)$ a
\emph{similitude} with multiplier $\mu$. In particular, 
equivalences of categories with duality yield equivalences of
the corresponding hermitian categories, which preserve
hyperbolicity. This applies to the morphisms 
from~\S\ref{subsec:periodual}, which thus define equivalences
\[
    \xymatrix{\catS(\catC,D,\pm\delta)
  \ar@<0.5ex>[r]^-{\catS(\widehat\Psi)}
  &
\ar@<0.5ex>[l]^-{\catS(\widecheck\Theta)}
\catS(\catC'',\conjt''D'',\pm\delta_{\alpha''}'')}
\quad\text{and}\quad
\xymatrix{\catS(\catC,\conjt D,\pm\delta_\alpha)
  \ar@<0.5ex>[r]^-{\catS(\widecheck\Psi)}
  &
\ar@<0.5ex>[l]^-{\catS(\widehat\Theta)}
\catS(\catC'',D'',\mp\delta'')}.
\]

Applying the $\catS$ construction to the
diagram~\eqref{eq:commdiagcatdual}, we obtain a diagram that commutes
up to similitudes:
\begin{equation}
  \label{eq:commdiagsymsp}
  \begin{split}
  \xymatrix{
    \catS(\catC,D,\delta)\ar[r]^-{\catS(\widehat F)}\ar@{=}[d]
    &
    \catS(\catC',D',\delta')\ar[r]^-{\catS(\widehat F')}\ar@{=}[d]
    &
    \catS(\catC'',D'',\delta'')\ar[d]_{\catS(\widehat \Theta)}
    \ar[r]^-{\catS(\widehat F'')}
    &
    \catS(\catC''',D''',\delta''')\ar[d]_{\catS(\widehat \Theta')}
    \ar[r]^-{\catS(\widehat F''')}
    &
    \cdots
    \\
    \catS(\catC,D,\delta)\ar[r]^-{\catS(\widehat F)}
    &
    \catS(\catC',D',\delta')\ar[r]^-{\catS(\widehat G)}
    &
    \catS(\catC,\conjt D,-\delta_\alpha)\ar[r]^-{\catS(\widecheck F)}
    &
    \catS(\catC',\conjt'D',-\delta'_{\alpha'})
    \ar[r]^-{\catS(\widecheck G)}
    &
    \cdots
  }
  \end{split}
\end{equation}
The $8$-periodic lower sequence folds into an oriented octagon of
hermitian categories:
\begin{equation}
  \label{eq:octasspaces}
  \begin{split}
  \xymatrix{
    \catS(\catC,D,\delta)\ar[r]^-{\catS(\widehat F)}&
    \catS(\catC',D',\delta')\ar[r]^-{\catS(\widehat G)}&
    \catS(\catC,\conjt D,-\delta_\alpha)\ar[d]^{\catS(\widecheck F)}
    \\
    \catS(\catC',\conjt'D',\delta'_{\alpha'})\ar[u]^{\catS(\widecheck G)}
    &
    &
    \catS(\catC',\conjt'D',-\delta'_{\alpha'})
    \ar[d]^{\catS(\widecheck G)}
    \\
    \catS(\catC,\conjt D,\delta_\alpha)\ar[u]^{\catS(\widecheck F)}
    &
    \catS(\catC',D',-\delta')\ar[l]_-{\catS(\widehat G)}
    &
    \catS(\catC,D,-\delta)\ar[l]_-{\catS(\widehat F)}
  }
  \end{split}
\end{equation}

The main result of this section is as follows:

\begin{theorem}
  \label{thm:octamain}
    Under the composition of two consecutive functors in the
  octagon~\eqref{eq:octasspaces}, the image of each category consists
  of hyperbolic symmetric spaces.
\end{theorem}

Since hyperbolicity is preserved by similitude
and by $\catS(\widehat\Theta)$ and $\catS(\widecheck\Theta)$, we may
substitute in the theorem the functors
in the upper row of~\eqref{eq:commdiagsymsp} for the 
functors $\catS(\widehat F)$, $\catS(\widehat G)$, $\catS(\widecheck
F)$, $\catS(\widecheck G)$ from~\eqref{eq:octasspaces}. Moreover,
since 
$\widehat F^{(n+1)}=\bigl(\widehat F^{(n)}\bigr)'$, we may substitute
$\widehat F$ for $\widehat F^{(n)}$ and consider only the composition
of $\catS(\widehat F)$ and $\catS(\widehat F')$. Therefore, the
theorem follows from the following proposition:

\begin{proposition}
  \label{prop:1}
  The composition $\catS(\widehat F')\catS(\widehat F) =
  \catS(\widehat F'\widehat F)$ maps every symmetric space in
  $\catS(\catC,D,\delta)$ to a hyperbolic symmetric space in
  $\catS(\catC'',D'',\delta'')$. 
\end{proposition}

The following lemma characterises the symmetric spaces in the image of
$\catS(\widehat F)$. It is to be used in the proofs of
Proposition~\ref{prop:1} and Theorems~\ref{thm:octasplit} and
\ref{thm:abelcat}. 

\begin{lemma}
  \label{lem:extend}
  For a symmetric space $(C',h')\in\catS(\catC',D',\delta')$, the
  following conditions are equivalent: 
  \begin{enumerate}
  \item[(a)]
    $(C',h')\simeq \catS(\widehat F)(X,h)$ for some symmetric space
    $(X,h)\in \catS(\catC,D,\delta)$;
  \item[(b)]
    there exists an isometry $\theta'\colon
    \bigl(\conjt'(C'),\alpha'_{C'}\circ \conjt'(h')\bigr) \to (C',h')$
    in $\catS(\catC',D',\delta')$ such that
    $\theta'\circ\conjt'(\theta') = \lambda'_{C'}$.
  \end{enumerate}
\end{lemma}

\begin{proof}
  (a)~$\Rightarrow$~(b) For every symmetric space
  $(X,h)\in\catS(\catC,D,\delta)$, the map $\theta':=\bigl(
  \begin{smallmatrix}
    0&\lambda_X\\ -\Idfunc_{\conjt(X)}&0
  \end{smallmatrix}
  \bigr)$ defines an isometry $\bigl(\conjt'F(X),\alpha'_{F(X)}\circ
    \conjt'(\widehat\phi_X\circ F(h))\bigr) \to
    \bigl(F(X),\widehat\phi_X\circ F(h)\bigr)$ such that
    $\theta'\circ\conjt'(\theta')=\lambda'_{F(X)}$. 
  \medbreak

  \noindent (b)~$\Rightarrow$~(a) Let $C'=(C,f)$. For $\theta'$
  in~(b), the pair $C'':=(C',\theta')$ is an object in $\catC''$. The
  functor $\Theta\colon \catC''\to\catC$ from
  \S\ref{subsec:periodconj} yields an object $\Theta(C'')\in\catC$,
  which we denote simply by $X$, and morphisms $X\xrightarrow{t}
  C\xrightarrow{s} X$ in $\catC$ such that
  \[
    s\circ t= \Idfunc_X \qquad\text{and}\qquad
    t\circ s = \textstyle{\frac12}(\Idfunc_C-G(\theta')\circ f^{-1}).
  \]
  Note that $\theta'$ is an isomorphism in $\catC'$, hence
  $f\circ\conjt G(\theta')= - G(\theta')\circ\conjt(f)$, and therefore
  \begin{equation}
    \label{eq:extend1}
    f^{-1}\circ G(\theta') = G(\theta')\circ f^{-1}
    \qquad\text{and}\qquad
    f\circ\conjt(t\circ s)\circ f^{-1} =
    \textstyle{\frac12}(\Idfunc_C+G(\theta')\circ f^{-1}) =
    \Idfunc_C-t\circ s.
  \end{equation}
  
  Define $h\colon X\to D(X)$ by $h=D(t)\circ G(h')\circ t$. Since
  $h'=D'(h')\circ \delta'_{C'}$ and $G(\delta'_{C'}) = \delta_C$, it
  follows that $h=D(h)\circ \delta_X$, but we need to show that $h$ is
  an isomorphism to be able to consider
  $(X,h)\in\catS(\catC,D,\delta)$.

  Recall from~\eqref{eq:etadef} the natural isomorphism $\eta\colon
  \Idfunc_{\catC''}\natiso \Psi\Theta$, which yields an isomorphism
  $\eta_{C''}= \bigl( 
  \begin{smallmatrix}
    s\\ \conjt(s)f^{-1}
  \end{smallmatrix}
  \bigr)\colon C''\to \Psi(X)$, hence also an isomorphism
  $
    G'(\eta_{C''})\colon C'\to G'\Psi(X) = F(X).
    $
    We claim that
    \begin{equation}
      \label{eq:extend2}
      h'=D'G'(\eta_{C''})\circ\widehat\phi_X\circ F(h)\circ
    G'(\eta_{C''}).
  \end{equation}
  This equation shows at once that $h$ is an isomorphism and that
  $G'(\eta_{C''})$ is an isometry $(C',h')\to \catS(\widehat F)(X,h)$,
  completing the proof of the lemma.

  To prove~\eqref{eq:extend2}, observe that the isometry $\theta'$
  induces an isometry $G(\theta')\colon (\conjt(C),-\alpha_C\circ
  \conjt G(h'))\to (C,G(h'))$ in $\catS(\catC,D,\delta)$, and that the
  condition that $h'$ is an isometry in $\catS(\catC',D',\delta')$
  implies that $f^{-1}$ is an isometry $(C,G(h')) \to
  (\conjt(C),-\alpha_C\circ \conjt G(h'))$ in
  $\catS(\catC,D,\delta)$. Therefore, $G(\theta')\circ f^{-1}$ is an
  isometry $(C,G(h'))\to (C,G(h'))$, and it follows that
  \begin{equation}
    \label{eq:extend3}
    D(t\circ s)\circ G(h')\circ t\circ s = G(h')\circ t\circ s.
  \end{equation}
  From~\eqref{eq:extend1} and \eqref{eq:extend3}, we obtain
  \[
    \conjt(D(t\circ s)\circ G(h')\circ t\circ s) \circ f^{-1} = \conjt
    G(h')\circ f^{-1} \circ (\Idfunc_C-t\circ s).
  \]
  Therefore, since $f^{-1}\colon (C,G(h'))\to
  (\conjt(C),-\alpha_C\circ \conjt G(h'))$ is an isometry,
  \begin{equation}
    \label{eq:extend4}
    D(f^{-1})\circ\alpha_C\circ \conjt(D(t\circ s)\circ G(h')\circ
    t\circ s) \circ f^{-1}= -G(h')\circ(\Idfunc_C-t\circ s).
  \end{equation}
  Now, by definition of $\eta_{C''}$ we have
  \[
    D'G'(\eta_{C''})\circ\widehat\phi_X\circ F(h)\circ
    G'(\eta_{C''}) = D(t\circ s)\circ G(h')\circ t\circ s -
    D(f^{-1})\circ\alpha_C\circ \conjt(D(t\circ s)\circ G(h')\circ 
    t\circ s) \circ f^{-1}.
  \]
  Therefore, \eqref{eq:extend2} follows from~\eqref{eq:extend3} and
  \eqref{eq:extend4}. 
\end{proof}

\begin{proof}[Proof of Proposition~\ref{prop:1}]
  Let $(C',h')\in\catS(\catC',D',\delta')$ be such that
  $(C',h')=\catS(\widehat F)(X,h)$ for some
  $(X,h)\in\catS(\catC,D,\delta)$. Lemma~\ref{lem:extend} yields an
  isometry $\theta'\colon\bigl(\conjt'(C'),\alpha'_{C'}\circ
  \conjt'(h')\bigr) \to (C',h')$ such that
  $\theta'\circ\conjt'(\theta') = \lambda'_{C'}$, hence an object
  $C''=(C',\theta')\in \catC''$. In $\catC''$, the morphism
  \[
    \begin{pmatrix}
      \Idfunc_{C'}& \theta'\\
      \textstyle{\frac12}h'&-\textstyle{\frac12}h'\circ\theta'
    \end{pmatrix}
    \colon F'(C')\to C''\oplus D''(C'')
  \]
  is an isometry $\catS(\widehat F')(C',h')\to H(C'')$.
\end{proof}

\subsection{Example: Hermitian spaces}
\label{subsec:exsymsp1}

For the typical dualities on categories of projective modules
discussed in 
\S\ref{subsec:exdual}, the symmetric spaces are (slightly generalised)
hermitian spaces. Let $(\sigma, \varepsilon)$ be an antistructure on a
$k$-algebra $A$ and let $(D,\delta)$ be the derived duality on
$\catP_A$. To every $(P,h)\in\catS(\catP_A,D,\delta)$ we attach a map
$P\times P\to A$, for which we 
use the same notation $h$, defined by
\[
  h(x,y)=\langle h(x),y\rangle
  \quad\text{for $x$, $y\in P$.}
\]
It readily follows from the definitions that $h$ is a
nonsingular sesquilinear map for the antiautomorphism $\sigma$ on $A$,
and the property $h=D(h)\circ\delta_P$ means that
\[
  h(y,x) = \sigma\bigl(h(x,y)\bigr)\,\varepsilon
  \qquad\text{for $x$, $y\in P$.}
\]
In C.T.C.~Wall's terminology \cite{WallI} the map $h\colon P\times
P\to A$ is a nonsingular \emph{$(\sigma,\varepsilon)$-reflexive
  form}. When $\varepsilon$ lies in the center of $A$, then $\sigma$
is an involution and $\catS(\catP_A,D,\delta)$ is the category of
$\varepsilon$-hermitian $A$-modules as defined in~\cite[p.~12]{Knus}
(and denoted there by $\mathfrak{H}^\varepsilon(A)$). Abusing
terminology, we also call \emph{$(\sigma,\varepsilon)$-hermitian
  modules} the symmetric spaces in $\catS(\catP_A,D,\delta)$  when
$(D,\delta)$ is derived from the antistructure $(\sigma,\varepsilon)$,
and call \emph{$(\sigma,\varepsilon)$-hermitian forms} C.T.C.~Wall's
$(\sigma,\varepsilon)$-reflexive forms.
We write
\[
  \catS(\catP_A,D,\delta)=\catH^{\varepsilon}(A,\sigma)
  \qquad\text{and}\qquad
  \catS(\catP_A,D,-\delta)=\catH^{-\varepsilon}(A,\sigma).
\]

Similarly, given a commutation $\alpha$ between $(\conjt,\lambda)$
and $(D,\delta)$ as in~\S\ref{subsec:exdual}, we attach to every
symmetric space $(P,h)\in\catS(\catP_A,\conjt D, \delta_\alpha)$ a map
\[
  h\colon P\times P\to A
  \quad\text{defined by}\quad
  h(x,y)=\{h(x),y\}
  \quad\text{for $x$, $y\in P$.}
\]
This map is a nonsingular $(\sigma\conjt,\varepsilon_\alpha)$-hermitian
form, hence
\[
  \catS(\catP_A,\conjt D,\delta_\alpha) =
  \catH^{\varepsilon_\alpha}(A,\sigma\conjt)
  \qquad\text{and}\qquad
  \catS(\catP_A,\conjt D,-\delta_\alpha) =
  \catH^{-\varepsilon_\alpha}(A,\sigma\conjt).
\]
Likewise, for the dualities $(D',\pm\delta')$ and
$(\conjt'D',\pm\delta'_{\alpha'})$ derived from
$(\sigma',\pm\varepsilon')$ 
and $(\sigma'\conjt',\pm\varepsilon_\alpha)$,
\[
  \catS(\catP_{A'},D',\pm\delta')=\catH^{\pm\varepsilon}(A',\sigma')
  \qquad\text{and}\qquad
  \catS(\catP_{A'},\conjt'D',\pm\delta'_{\alpha'}) =
  \catH^{\pm\varepsilon_\alpha}(A',\sigma'\conjt').
\]

The functors
\[
  \catS(\widehat F)\colon \catH^{\pm\varepsilon}(A,\sigma) \to
  \catH^{\pm\varepsilon}(A',\sigma')
  \quad\text{and}\quad \catS(\widecheck F)\colon
\catH^{\pm\varepsilon_\alpha}(A,\sigma\conjt) \to
\catH^{\pm\varepsilon_\alpha}(A',\sigma'\conjt')
\]
are given by scalar extensions, mapping
every $(\sigma,\pm\varepsilon)$- or $(\sigma\conjt,\pm
\varepsilon_\alpha)$-hermitian space $(P,h)$ to
$(P\otimes_AA',h')$ defined by
\begin{multline*}
  h'(x_1\otimes1+x_2\otimes j,\,y_1\otimes1+y_2\otimes j) =
  h(x_1,y_1) + \sigma'(j)h(x_2,y_2)j +
  h(x_1,y_2)j + \sigma'(j) h(x_2,y_1)
  \\
  = h(x_1,y_1) - u\,\conjt^{-1}(h(x_2,y_2))
  +\bigl(h(x_1,y_2) - u\,\conjt^{-1}(
  h(x_2,y_1))z^{-1}\bigr) j
\end{multline*}
for $x_1$, $x_2$, $y_1$, $y_2\in P$.

On the other hand, the functors
\[
  \catS(\widehat G)\colon
\catH^{\pm\varepsilon}(A',\sigma')\to
\catH^{\mp\varepsilon_\alpha}(A,\sigma\conjt)
\quad\text{and}\quad \catS(\widecheck
G)\colon \catH^{\pm\varepsilon_\alpha}(A',\sigma'\conjt') \to
\catH^{\pm\varepsilon}(A,\sigma)
\]
are given by transfer maps. To make this
clear, consider the map
\[
  \tr\colon A'\to A
  \quad\text{defined by} \quad
  \tr(a_1+a_2j) = a_1
  \quad\text{for $a_1$, $a_2\in A$.}
\]
It satisfies $\tr(ja)=\conjt\bigl(\tr(aj)\bigr)$,
$\tr\bigl(\conjt'(a)\bigr)=\conjt\bigl(\tr(a)\bigr)$ and
$\tr\bigl(\sigma'(a)\bigr)=\sigma\bigl(\tr(a)\bigr)$ for 
all $a\in A'$. Moreover, for $P'\in\catP_{A'}$, the identifications
$GD'(P')=DG'(P')$ and $D'(P')=\Hom_{A'}(P',A')$
from~\eqref{eq:pairing'} show that
\[
  \langle \pi, x\rangle = \tr(\langle\pi,x\rangle')
  \qquad\text{for $\pi\in D'(P')$ and $x\in P'$.}
\]
Let $(P',h')\in\catS(\catP_{A'},D',\pm\delta')$ and
$h=\widehat\gamma_{P'}\circ G(h')$, so $(G(P'),h)$ is the image of
$(P',h')$ under $\catS(\widehat G)$. Let also sesquilinear maps $h'$ and
$h$ be defined by
\[
  h'(x,y)=\langle h'(x), y\rangle' \quad\text{and}\quad
  h(x,y) = \{ h(x), y\}
  \quad\text{for $x$, $y\in P'$.}
\]
The definition of $\alpha_{G(P')}$ in~\eqref{eq:defalphaP} yields
\[
  \langle\alpha_{G(P')}(h(x)),{}^\conjt y\rangle = u\;
  \conjt^{-1}\bigl(h(x,y)\bigr) \qquad\text{for $x$, $y\in
    P'$.}
\]
On the other hand, from the definition of $\widehat\gamma_{P'}$ it
follows that
\[
  \langle\alpha_{G(P')}(h(x)),{}^\conjt y\rangle = \langle h'(x),
  yj\rangle = \tr(\langle h'(x),yj\rangle').
\]
Therefore,
\[
  h(x,y) = \conjt(u)^{-1} \conjt\bigl(\tr(
  h'(x,y)j)\bigr) = \conjt(u)^{-1}\tr(j\,h'(x,y))
  \quad\text{for $x$, $y\in P'$.}
\]
Similarly, for
$(P',h')\in\catS(\catP_{A'},\conjt'D',\pm\delta'_{\alpha'})$ and 
$(G(P'),h)$ with $h=\widecheck\gamma_{P'}\circ G(h')$ the image of
$(P',h')$ under $\catS(\widecheck G)$, the associated
$(\sigma,\pm\varepsilon)$-hermitian and $(\sigma'\conjt',\pm
\varepsilon_\alpha)$-hermitian forms $h$ and
$h'$ are related by
\[
  h(x,y) = u\;\tr(j^{-1}h'(x,y))
  \qquad\text{for $x$, $y\in P'$.}
\]
The functors $\catS(\widehat G)$ and $\catS(\widecheck G)$ are thus
given by the transfers $\widehat\tr$ and $\widecheck\tr$ defined by
\[
  \widehat\tr(h')(x,y) =
  \conjt(u)^{-1}\tr(j\,h'(x,y)) \qquad\text{for
    $(P', h')\in\catH^{\pm\varepsilon}(A',\sigma')$ and $x$,
    $y\in P'$}
\]
and
\[
  \widecheck\tr(h')(x,y) = u\;\tr(j^{-1}
  h'(x,y)) \qquad\text{for $(P',
    h')\in\catH^{\pm\varepsilon_\alpha}(A',\sigma'\conjt')$ and $x$,
    $y\in P'$.}
\]
Writing simply $\ext$ for the scalar extension functors
$\catS(\widehat F)$ and $\catS(\widecheck F)$, the
octagon~\eqref{eq:octasspaces} takes the form
\begin{equation}
  \label{eq:octaherm}
  \begin{split}
  \xymatrix{
    \catH^{\varepsilon}(A,\sigma)\ar[r]^-{\ext}&
    \catH^{\varepsilon}(A',\sigma')\ar[r]^-{\widehat\tr}&
    \catH^{-\varepsilon_\alpha}(A,\sigma\conjt)\ar[d]^{\ext}
    \\
    \catH^{\varepsilon_\alpha}(A',\sigma'\conjt')\ar[u]^{\widecheck\tr}
    &
    &
    \catH^{-\varepsilon_\alpha}(A',\sigma'\conjt')
    \ar[d]^{\widecheck\tr}
    \\
    \catH^{\varepsilon_\alpha}(A,\sigma\conjt)\ar[u]^{\ext}
    &
    \catH^{-\varepsilon}(A',\sigma')\ar[l]_-{\widehat\tr}
    &
    \catH^{-\varepsilon}(A,\sigma)\ar[l]_-{\ext}
  }
  \end{split}
\end{equation}

\subsection{Example: The Grenier-Boley--Mahmoudi--First octagon}
\label{subsec:GBMF}

We spell out in this subsection the relation
between~\eqref{eq:octaherm} and the octagon from~\cite{GBM} and
\cite{First}. Let $A$ be a $k$-algebra with a $k$-linear involution
$\sigma$ and let $\varepsilon\in A^\times$ be a central unit such that
$\sigma(\varepsilon)=\varepsilon^{-1}$. For every unit $\xi\in
A^\times$ we write $\Int(\xi)$ for the inner automorphism mapping
$a\in A$ to $\xi a\xi^{-1}$. Assume $A$ contains units $\zeta$, $\eta$
such that
\[
  \zeta^2\in k,\qquad \sigma(\zeta)=-\zeta, \qquad
  \sigma(\eta)=-\eta\quad\text{and}\quad
  \zeta\eta=-\eta\zeta.
\]
Let $B=C_A(\zeta)$ be the centraliser of $\zeta$ and let
\[
  \lambda_B=\eta^2\in B,\qquad
  \conjt_B=\Int(\eta)\rvert_B,\qquad
  \tau_1=\sigma\rvert_B \quad\text{and}\quad
  \tau_2=\tau_1\circ\conjt_B.
\]
The pair $(\conjt_B,\lambda_B)$ is a structure on $B$ in the sense
of~\S\ref{subsec:exconjug1}, and the pair $(\tau_1,\varepsilon)$ is an
antistructure on $B$. The unit $\lambda_B\in B^\times$ satisfies the
conditions~(i)---(iii) on $u$ in~\S\ref{subsec:exdual}, hence it
defines as in~\eqref{eq:defalphaP} a commutation $\alpha_B$ between the
conjugation derived from~$(\conjt_B,\lambda_B)$ and the duality on
$\catP_B$ derived from $(\tau_1,\varepsilon)$. The decomposition
\[
    a=\textstyle{\frac12}(a+\zeta a\zeta^{-1}) +
    \textstyle{\frac12}(a-\zeta a\zeta^{-1})
    \qquad\text{for $a\in A$}
\]
yields $A=B\oplus B\eta$, hence $A$ is the
$(\conjt_B,\lambda_B)$-twisted quadratic extension $B'$ of $B$ and
$\sigma$ is the extension $\tau_1'$ of $\tau_1$ to $B'=A$. Moreover,
$\Int(\eta\zeta)$ is the extension $\conjt'_B$ of $\conjt_B$ to $A$,
and $\varepsilon_{\alpha_B}=\varepsilon$. With these data, the
octagon~\eqref{eq:octaherm} becomes
\begin{equation}
  \label{eq:octahermGBMF}
  \begin{split}
  \xymatrix{
    \catH^{\varepsilon}(B,\tau_1)\ar[r]^-{\ext}&
    \catH^{\varepsilon}(A,\sigma)\ar[r]^-{\widehat\tr}&
    \catH^{-\varepsilon}(B,\tau_2)\ar[d]^{\ext}
    \\
    \catH^{\varepsilon}(A,\sigma\circ\Int(\eta\zeta))\ar[u]^{\widecheck\tr}
    &
    &
    \catH^{-\varepsilon}(A,\sigma\circ\Int(\eta\zeta))
    \ar[d]^{\widecheck\tr}
    \\
    \catH^{\varepsilon}(B,\tau_2)\ar[u]^{\ext}
    &
    \catH^{-\varepsilon}(A,\sigma)\ar[l]_-{\widehat\tr}
    &
    \catH^{-\varepsilon}(B,\tau_1)\ar[l]_-{\ext}
  }
  \end{split}
\end{equation}
This octagon is only superficially different from the one
in~\cite[(3.1)]{First}, which has $\catH^{\mp\varepsilon}(B,\tau_1)$
for $\catH^{\pm\varepsilon}(B,\tau_1)$ and
$\catH^{\pm\varepsilon}(A,\sigma)$ for
$\catH^{\mp\varepsilon}(A,\sigma\circ\Int(\eta\zeta))$. The
differences are resolved by scaling: if
$(Q,g)\in\catH^{\pm\varepsilon}(B,\tau_1)$, 
then $(Q,\zeta g)\in\catH^{\mp\varepsilon}(B,\tau_1)$, and if
$(P,h)\in\catH^{\mp\varepsilon}(A,\sigma\circ\Int(\eta\zeta))$, then
$(P,\zeta\eta h)\in\catH^{\pm\varepsilon}(A,\sigma)$.

Note that, by the very nature of its periodic construction, the same
octagon can still be obtained with a shift of one edge by starting
with the algebra $A$. To see this explicitly, let $\conjt_A=\Idfunc_A$
and $\lambda_A=\zeta^2\in k^\times$, and set $K=k\oplus kj$ with
$j^2=\lambda_A$, a quadratic \'etale $k$-algebra with nontrivial
automorphism $\can_{K/k}$. The $(\conjt_A,\lambda_A)$-twisted
quadratic extension of $A$ is $A'=A\otimes_kK$, with the automorphism
$\conjt'_A=\Idfunc_A\otimes\can_{K/k}$. The scalar $\lambda_A$
satisfies the conditions~(i)---(iii) on $u$ in~\S\ref{subsec:exdual},
hence it yields a commutation $\alpha_A$ between the conjugation on
$\catP_A$ derived from $(\conjt_A,\lambda_A)$ and the duality derived
from $(\sigma,\varepsilon)$. With this choice,
$\varepsilon_{\alpha_A}=\varepsilon$ and the involutions $\sigma'$ and
$\sigma'\conjt'_A$ on $A'$ are $\sigma\otimes\can_{K/k}$ and
$\sigma\otimes\Idfunc_K$ respectively. In order to relate the
resulting octagon~\eqref{eq:octaherm} with~\eqref{eq:octahermGBMF} or
\cite[(3.1)]{First}, observe that $e=\frac12(1+\zeta\otimes j^{-1})\in
A'$ is an idempotent centralised by $B$. Mapping $b\in B$ to
$e(b\otimes1)e$ identifies $B$ with
$eA'e\simeq\End_{A'}(eA')$. Following~\cite[I(9.1)]{Knus}, 
Morita theory yields an equivalence of categories
$\catP_B\xrightarrow{\sim}\catP_{A'}$ mapping $P\in\catP_B$ to
$P\otimes_BeA'$. Moreover, equivalences of categories
\[
  \Mor_1\colon\catH^{\pm\varepsilon}(B,\tau_1) \xrightarrow{\sim}
  \catH^{\pm\varepsilon}(A',\sigma') \qquad\text{and}\qquad
  \Mor_2\colon\catH^{\pm\varepsilon}(B,\tau_2) \xrightarrow{\sim}
  \catH^{\mp\varepsilon}(A',\sigma'\conjt')
\]
can be obtained by hermitian form composition with the maps $h_1$,
$h_2\colon eA'\times eA'\to A'$ 
defined by
\[
  h_1(x,y)=\sigma'(x)y \qquad\text{and}\qquad
  h_2(x,y)=\sigma'\conjt'(x)\eta y\qquad\text{for $x$, $y\in eA'$}.
\]
The map $h_1$ is a nonsingular hermitian form over $(A',\sigma')$ and
$h_2$ is a nonsingular 
skew-hermitian form over $(A',\sigma'\conjt'_A)$.

The following diagram illustrates the relation between First's octagon
in~\cite[(3.1)]{First} 
(consisting of the outer maps) and the octagon of~\eqref{eq:octaherm}
(of inner maps) with the data above. All the triangles commute up to similitude.

\begin{equation*}
  \label{eq:octafirst}
  \begin{split}
    \xymatrix{
      &&\catH^\varepsilon(B,\tau_1)\ar[dr]^{\rho_1^{(\varepsilon)}}
      \ar[d]^{\Mor_1}&&
      \\
      &
      \catH^{\varepsilon}(A,\sigma)\ar[r]^-{\ext}
      \ar[ur]^{\pi_1^{(\varepsilon)}}
      &
    \catH^{\varepsilon}(A',\sigma')\ar[r]^-{\widehat\tr}&
    \catH^{-\varepsilon}(A,\sigma)\ar[d]^{\ext}
    \ar[dr]^{\pi_2^{(-\varepsilon)}}
    &
    \\
    \catH^{-\varepsilon}(B,\tau_2)\ar[ur]^{\rho_2^{(-\varepsilon)}}
    \ar[r]^{\Mor_2}&
    \catH^{\varepsilon}(A',\sigma'\conjt')\ar[u]^{\widecheck\tr}
    &
    &
    \catH^{-\varepsilon}(A',\sigma'\conjt')
    \ar[d]^{\widecheck\tr}
    &
    \catH^\varepsilon(B,\tau_2)\ar[dl]^{\rho_2^{(\varepsilon)}}
    \ar[l]_-{\Mor_2}
    \\
    &
    \catH^{\varepsilon}(A,\sigma)\ar[u]^{\ext}
    \ar[ul]^{\pi_2^{(\varepsilon)}}
    &
    \catH^{-\varepsilon}(A',\sigma')\ar[l]_-{\widehat\tr}
    &
    \catH^{-\varepsilon}(A,\sigma)\ar[l]_-{\ext}
    \ar[dl]^{\pi_1^{(-\varepsilon)}}
    &
    \\
    && \catH^{-\varepsilon}(B,\tau_1)
    \ar[ul]^{\rho_1^{(-\varepsilon)}}
    \ar[u]^{\Mor_1}&&
  }
  \end{split}
\end{equation*}

\section{Witt groups}
\label{sec:Witt}

Throughout this section, $\catC$ is, as in the preceding sections, a
pseudo-abelian $k$-linear category where $k$ is a commutative ring in
which $2$ is invertible. We assume moreover that $\catC$ is
essentially small, which means that the isomorphism classes of objects
in $\catC$ form a set. If $\catC$ carries a conjugation
$(\conjt,\lambda)$ or a duality $(D,\delta)$, it follows that the
categories $\catC'$ and $\catS(\catC,D,\delta)$ also are essentially
small. If $(\catC,D,\delta)$ is an \emph{exact} category with duality
in the sense of~\cite[\S1.2.2]{Balmer}, then \emph{metabolic}
symmetric spaces are defined, and a classical construction due to
Knebusch \cite[Def.~27]{Balmer} yields the \emph{Witt group}
$W(\catC,D,\delta)$: this group is the quotient of the monoid of
isometry classes of symmetric spaces in $\catS(\catC,D,\delta)$ by the
submonoid of isometry classes of metabolic symmetric spaces. Thus, the
definition of $W(\catC,D,\delta)$ depends on additional structure 
on the category with duality $(\catC,D,\delta)$.

If $(\catC_1,D_1,\delta_1)$ and $(\catC_2,D_2,\delta_2)$ are exact
categories with duality and $\widetilde R=(R,\theta)\colon
(\catC_1,D_1,\delta_1)\to (\catC_2,D_2,\delta_2)$ is a morphism of
categories with duality with $R$ an exact functor, then it follows
from the definitions that the functor $\catS(\widetilde R)\colon
\catS(\catC_1,D_1,\delta_1)\to \catS(\catC_2,D_2,\delta_2)$ induces a
group homomorphism $W(\widetilde R)\colon W(\catC_1,D_1,\delta_1) \to
W(\catC_2,D_2,\delta_2)$; and if two morphisms $\widetilde R$,
$\widetilde S\colon (\catC_1,D_1,\delta_1)\to (\catC_2,D_2,\delta_2)$
preserving the exact structures are related by a similitude
$\rho\colon \widetilde R \natiso \widetilde S$ with multiplier $\mu\in
k^\times$, then the similitude $\catS(\rho)\colon \catS(\widetilde R)
\natiso \catS(\widetilde S)$ defines an isomorphism
$\langle\mu\rangle$ mapping the Witt class of
$(C,h)\in\catS(C_2,D_2,\delta_2)$ to the Witt class of $(C,\mu h)$,
which fits in 
the following commutative diagram:
\[
  \xymatrix{
    W(\catC_1,D_1,\delta_1)\ar[r]^{W(\widetilde R)}
    \ar[dr]_{W(\widetilde S)}
    &
    W(\catC_2,D_2,\delta_2)\ar[d]^{\langle\mu\rangle}
    \\
    & W(\catC_2,D_2,\delta_2)
  }
\]

In this section, we consider two cases where Knebusch's construction
turns the octagon~\eqref{eq:octasspaces} of hermitian categories into
an $8$-periodic chain complex of Witt groups: the split exact
structure and the usual exact structure in abelian categories. In each
case, the crucial property that determines the exactness of the
resulting octagon of Witt groups is the following:

\begin{property}
  Let $(C,h)\in\catS(\catC,D,\delta)$ be a symmetric space. If
  $\catS(\widehat F)(C,h)$ is metabolic, then there exists an isometry
  $\theta\colon \bigl(\conjt(C),\alpha_C\circ\conjt(h)\bigr)\to (C,h)$
  such that 
  $\theta\circ\conjt(\theta)=\lambda_C$. 
\end{property}

When this property holds for representatives of each Witt class in the
Witt groups involved in the octagon, exactness
follows from Lemma~\ref{lem:extend}.

\begin{remark}
  \label{rem:WCP}
  Note that $\catS(\widehat F)(C,h) = (F(C), \widehat\phi_C\circ
  F(h))$ with
  $(GF(C),G(\widehat\phi_C\circ F(h))) =
  (C,h)\perp(\conjt(C), -\alpha_C\circ\conjt(h))$. When Witt
  cancellation holds in $\catS(\catC,D,\delta)$, there exists an
  isometry from $(\conjt(C), \alpha_C\conjt(h))$ to $(C,h)$
  if and only if $(C,h)\perp(\conjt(C), -\alpha_C\circ\conjt(h))\simeq
  (C,h)\perp (C,-h)$, i.e., $G\catS(\widehat F)(C,h) \simeq H(C)$.
\end{remark}

\subsection{The split exact structure}
\label{subsec:split}

In this case, the exact sequences are the split sequences, hence every
additive functor preserves the exact structure. Since
$2$ is invertible in $k$, the metabolic spaces are 
the hyperbolic spaces, see~\cite[Example~21]{Balmer}. Therefore, the
Witt group $W(\catC,D,\delta)$ is the 
usual Witt group as defined in~\cite[Ch.~7, \S2]{Scharlau}.
When $\catC$ carries a conjugation $(\conjt,\lambda)$ and a commuting
duality $(D,\delta)$, the
functors in the octagon~\eqref{eq:octasspaces} induce Witt group
homomorphisms since they carry hyperbolic spaces to hyperbolic
spaces, and Theorem~\ref{thm:octamain} shows that they form an
$8$-periodic chain complex of Witt groups:
\begin{equation}
  \label{eq:octaWitt}
  \begin{split}
  \xymatrix{
    W(\catC,D,\delta)\ar[r]^-{W(\widehat F)}&
    W(\catC',D',\delta')\ar[r]^-{W(\widehat G)}&
    W(\catC,\conjt D,-\delta_\alpha)\ar[d]^{W(\widecheck F)}
    \\
    W(\catC',\conjt'D',\delta'_{\alpha'})\ar[u]^{W(\widecheck G)}
    &
    &
    W(\catC',\conjt'D',-\delta'_{\alpha'})
    \ar[d]^{W(\widecheck G)}
    \\
    W(\catC,\conjt D,\delta_\alpha)\ar[u]^{W(\widecheck F)}
    &
    W(\catC',D',-\delta')\ar[l]_-{W(\widehat G)}
    &
    W(\catC,D,-\delta)\ar[l]_-{W(\widehat F)}
  }
  \end{split}
\end{equation}
In this subsection, we provide conditions under which this chain
complex is an exact sequence.
\medbreak

For every symmetric space $(C,h)\in\catS(\catC,D,\delta)$ the
endomorphism $k$-algebra $\End_\catC(C)$ carries an involution
$\sigma_h$ (i.e., an anti-automorphism of period~$2$) defined by
\[
  \sigma_h(\varphi)=h^{-1}\circ D(\varphi)\circ h
  \qquad\text{for $\varphi\in\End_\catC(C)$}.
\]
The symmetric space $(C,h)$ is said to be \emph{split anisotropic} if
there is no nonzero direct summand $t\colon L\to C$ such that
$D(t)\circ h\circ t=0$ or, equivalently since $\catC$ is
pseudo-abelian, if there is no nonzero idempotent $e\in\End_\catC(C)$
such that $\sigma_h(e)\circ e=0$. Recall that $\End_\catC(C)$ is said
to be \emph{semilocal} if factoring out its Jacobson radical $R$
yields a $k$-algebra $\End_\catC(C)/R$ that is right artinian; it is
\emph{rad-adically complete} if it is complete for the $R$-adic
topology (in which $(R^n)_{n\in\mathbb{N}}$ is a basis of
neighborhoods of $0$), see~\cite[II(4.5)]{Knus} or
\cite[\S21]{Lam}. For example, right artinian algebras are semilocal
and rad-adically complete since their Jacobson radical is nilpotent,
see~\cite[(4.12)]{Lam}. 

The next lemma demonstrates how these properties are used in the proof
of Property~(P).

\begin{lemma}
  \label{lem:invert}
  Let $(C,h)\in\catS(\catC,D,\delta)$ be a split anisotropic symmetric
  space. If $\End_\catC(C)$ is semilocal and rad-adically complete,
  then every $\varphi\in\End_\catC(C)$ such that
  $\sigma_h(\varphi)=\Idfunc_C-\varphi$ is invertible.
\end{lemma}

\begin{proof}
  Let $R$ be the Jacobson radical of $\End_\catC(C)$ and
  $\overline{\End_\catC(C)}=\End_\catC(C)/R$. We write
  $\overline\varphi$ for the image of $\varphi\in\End_\catC(C)$ in
  $\overline{\End_\catC(C)}$. Properties of the Jacobson radical imply
  that $\varphi$ is invertible if $\overline\varphi$ is invertible
  (see~\cite[II(4.2.2)]{Knus} or \cite[(4.8)]{Lam}), and that
  idempotents in $\overline{\End_\catC(C)}$ lift to idempotents in
  $\End_\catC(C)$ when $\End_\catC(C)$ is rad-adically complete
  (see~\cite[II(4.5.4)]{Knus} or \cite[(21.31)]{Lam}). Since
  $\sigma_h$ is an anti-automorphism, it preserves $R$ and induces an
  involution $\overline{\sigma_h}$ on $\overline{\End_\catC(C)}$.

    As a first step in the proof, we show that if
    $\varepsilon\in\overline{\End_\catC(C)}$ is an idempotent such
    that $\overline{\sigma_h}(\varepsilon)\varepsilon=0$, then
    $\varepsilon=0$. This follows from the hypothesis that $(C,h)$ is
    split anisotropic by the reduction theory in~\cite[II.4]{Knus} or
    \cite[Ch.~7, \S4]{Scharlau}; the following is an easy direct
    proof.

    Let $e_0\in\End_\catC(C)$ be an idempotent such that
    $\overline{e_0}=\varepsilon$. We construct a sequence of
    idempotents $(e_n)_{n\in\mathbb{N}}$ in $\End_\catC(C)$ such that
    $e_{n+1}\equiv e_n\bmod R^{n+1}$ and $\sigma_h(e_n)\circ e_n \in
    R^{n+1}$ for all $n\in\mathbb{N}$: if $e_n$ is given, then
    $\Idfunc_C+\frac12\sigma_h(e_n)\circ e_n$ is invertible since
    $\overline{\Idfunc_C+\frac12\sigma_h(e_n)e_n} =
    \overline{\Idfunc_C}$. We may therefore set
    \[
      e_{n+1} = (\Idfunc_C+\textstyle{\frac12}\sigma_h(e_n)\circ
      e_n)^{-1} \circ e_n \circ
      (\Idfunc_C+\textstyle{\frac12}\sigma_h(e_n)\circ e_n).
    \]
    This is clearly an idempotent, and $e_{n+1}\equiv e_n\bmod
    R^{n+1}$ since $\frac12\sigma_h(e_n)\circ e_n\in R^{n+1}$. This
    last property also implies
    \begin{multline*}
      e_{n+1}\equiv (\Idfunc_C-\textstyle{\frac12}\sigma_h(e_n)\circ
      e_n) \circ 
      e_n \circ (\Idfunc_C+\textstyle{\frac12}\sigma_h(e_n)\circ e_n)
      \\
      \equiv e_n- \textstyle{\frac12}\sigma_h(e_n)\circ e_n +
      \textstyle{\frac12} e_n\circ \sigma_h(e_n)\circ e_n \bmod
      R^{n+2},
    \end{multline*}
    whence $\sigma_h(e_{n+1})\circ e_{n+1} \in R^{n+2}$. Since
    $\End_\catC(C)$ is assumed to be $R$-adically complete, there
    exists $e_\infty\in\End_\catC(C)$ such that $e_\infty\equiv e_n
    \bmod R^{n+1}$ for all $n\in\mathbb{N}$. This endomorphism is
    idempotent and $\sigma_h(e_\infty)\circ e_\infty=0$. But $(C,h)$
    is assumed to be split anisotropic, hence $e_\infty=0$ and
    therefore $\varepsilon=0$.

    Now, suppose $\varphi\in\End_\catC(C)$ is such that
    $\sigma_h(\varphi) = \Idfunc_C-\varphi$, hence
    $\overline{\sigma_h}(\overline\varphi) = \overline{\Idfunc_C} -
    \overline\varphi$. The Wedderburn--Artin structure theorem
    (see~\cite[(3.5)]{Lam}) shows that $\overline{\End_\catC(C)}$ is
    isomorphic to a direct product of matrix algebras over division
    algebras. Therefore, $\overline\varphi$ is invertible unless there
    exists a nonzero idempotent
    $\varepsilon\in\overline{\End_\catC(C)}$ such that
    $\overline\varphi\,\varepsilon=0$. But if
    $\overline\varphi\,\varepsilon=0$, then
    \[
      \overline{\sigma_h}(\varepsilon)\varepsilon =
      \overline{\sigma_h}(\varepsilon)\,(\overline{\Idfunc_C} -
      \overline\varphi)\, \varepsilon =
      \overline{\sigma_h}(\overline\varphi\, \varepsilon)\varepsilon =
      0.
    \]
    We just saw that this implies $\varepsilon=0$, hence
    $\overline\varphi$ is invertible and $\varphi$ is invertible.
  \end{proof}

\begin{proposition}
  \label{prop:Psplit}
  Property~(P) holds for $(C,h)\in\catS(\catC,D,\delta)$ if $(C,h)$ is
  split anisotropic and $\End_\catC(C)$ is semilocal and rad-adically
  complete. 
\end{proposition}

\begin{proof}
  Let $\catS(\widehat F)(C,h) = (F(C),h')$ where
  $h'=\widehat\phi_C\circ F(h) = \bigl(
  \begin{smallmatrix}
    h&0\\ 0&-\alpha_C\circ \conjt(h)
  \end{smallmatrix}
  \bigr)$. Assume $(F(C),h')\simeq H(L')$ for some $L'\in\catC'$,
  hence $F(C)$ may be identified with $L'\oplus D'(L')$. Projection on
  $L'$ yields an 
  idempotent $e\in\End_{\catC'}(F(C))$ such that $\Idfunc_{F(C)}-e$ is
  the projection on $D'(L')$. The involution adjoint to the hyperbolic
  form on $L'\oplus D'(L')$ interchanges the projections on $L'$ and
  on $D'(L')$, hence $\sigma_{h'}(e)=\Idfunc_{F(C)}-e$. Viewing
  $\End_{\catC'}(F(C))\subset \End_\catC(C\oplus\conjt(C))$, we may
  write $e=\bigl(
  \begin{smallmatrix}
    \varphi_0&\varphi_1\\ \varphi_2&\varphi_3
  \end{smallmatrix}
  \bigr)$ for some $\varphi_0\in\End_\catC(C)$,
  $\varphi_1\in\Hom_\catC(\conjt(C),C)$,
  $\varphi_2\in\Hom_\catC(C,\conjt(C))$ and
  $\varphi_3\in\End_\catC(\conjt(C))$. Because $e$ is a morphism in
  $\catC'$, we have $\varphi_2=\conjt(\varphi_1)\circ\lambda_C^{-1}$
  and $\varphi_3=\conjt(\varphi_0)$. The equation $e^2=e$ then yields
  \begin{equation}
    \label{eq:P1}
    \varphi_0^2+\varphi_1\circ\conjt(\varphi_1)\circ\lambda_C^{-1} =
    \varphi_0 \quad\text{and}\quad
    \varphi_0\circ\varphi_1+\varphi_1\circ\conjt(\varphi_0) = \varphi_1.
  \end{equation}
  On the other hand, the condition $\sigma_{h'}(e)=\Idfunc_{F(C)}-e$
  yields
  \begin{equation}
    \label{eq:P2}
    \sigma_h(\varphi_0)=\Idfunc_C-\varphi_0 \quad\text{and}\quad
    h^{-1}\circ D(\conjt(\varphi_1)\circ\lambda_C^{-1})\circ
    \alpha_C\circ \conjt(h)=\varphi_1.
  \end{equation}
  Using the first equation in~\eqref{eq:P2}, we may rewrite the
  equations in~\eqref{eq:P1} as
  \begin{equation}
    \label{eq:P3}
    \varphi_1\circ\conjt(\varphi_1)\circ \lambda_C^{-1} =
    \sigma_h(\varphi_0)\circ \varphi_0 \quad\text{and}\quad
    \varphi_1\circ \conjt(\varphi_0) = \sigma_h(\varphi_0)\circ \varphi_1.
  \end{equation}
  Moreover, the first equation in~\eqref{eq:P2} shows by
  Lemma~\ref{lem:invert} that $\varphi_0$ is invertible, hence it
  follows from the first equation in~\eqref{eq:P3} that $\varphi_1$ is
  also invertible. We may therefore define an isomorphism
  $\theta\colon \conjt(C)\to C$ by
  $\theta=\sigma_h(\varphi_0)^{-1}\circ \varphi_1$. The equations
  in~\eqref{eq:P3} yield
  \[
    \theta = \varphi_0\circ \lambda_C\circ \conjt(\varphi_1)^{-1} =
    \varphi_1\circ \conjt(\varphi_0)^{-1}.
  \]
  Therefore, $\theta\circ \conjt(\theta) = \varphi_0\circ \lambda_C
  \circ \conjt(\varphi_1)^{-1}\circ \conjt(\varphi_1)\circ
  \conjt\conjt(\varphi_0)^{-1} = \lambda_C$. Moreover, the second
  equation in~\eqref{eq:P2} yields
  \[
    \alpha_C\circ\conjt(h) = D(\lambda_C\circ\conjt(\varphi_1)^{-1})
    \circ h \circ \varphi_1 = D(\varphi_0^{-1}\circ\theta)\circ h\circ
    \varphi_1.
  \]
  But, by definition, $\theta=\sigma_h(\varphi_0)^{-1}\circ \varphi_1
  = h^{-1}\circ D(\varphi_0)^{-1}\circ h\circ \varphi_1$, hence we
  obtain by substituting in the last equation
  \[
    \alpha_C\circ\conjt(h) = D(\theta)\circ h\circ\theta.
  \]
  This shows that $\theta$ is an isometry
  $(\conjt(C),\alpha_C\circ\conjt(h)) \to (C,h)$ and completes the proof.
\end{proof}

Next, we show that the hypotheses on endomorphism algebras in
Proposition~\ref{prop:Psplit} pass on to the double category.

\begin{proposition}
  \label{prop:Psplit2}
  Let $C'=(C,f)\in\catC'$. If $\End_\catC(C)$ is semilocal and
  rad-adically complete, then $\End_{\catC'}(C')$ also is semilocal
  and rad-adically complete.
\end{proposition}

\begin{proof}
  Let $\tau\colon \End_\catC(C)\to \End_\catC(C)$ be the automorphism
  defined by
  \[
    \tau(\varphi)=f\circ\conjt(\varphi)\circ f^{-1} \qquad\text{for
      $\varphi\in\End_\catC(C)$.}
  \]
  It satisfies $\tau^2(\varphi)=\varphi$ for all
  $\varphi\in\End_\catC(C)$ and
  \[
    \End_{\catC'}(C') = \{\varphi\in\End_\catC(C)\mid \tau(\varphi) =
    \varphi\}.
  \]
  Therefore, the proposition follows from~(b) in the next
  ring-theoretic result:
  \renewcommand{\qed}{\relax}
\end{proof}

\begin{lemma}
  \label{lem:Psplit2}
  Let $A$ be a $k$-algebra and $\tau\colon A\to A$ an automorphism
  such that $\tau^2=\Idfunc_A$, and set $B=\{a\in A\mid
  \tau(a)=a\}$. Write $\rad(A)$ (resp.\ $\rad(B)$) for the Jacobson
  radical of $A$ (resp.\ $B$).
  \begin{enumerate}
  \item[(a)]
    If $A$ is right artinian and $\rad(A)=0$, then $B$ also is right
    artinian and $\rad(B)=0$.
  \item[(b)]
    If $A$ is semilocal, then $\rad(B)=B\cap\rad(A)$ and $B$ is
    semilocal. If moreover $A$ is rad-adically complete, then $B$ also
    is rad-adically complete.
  \end{enumerate}
\end{lemma}

\begin{proof}
  (a)
  Setting $E=\{a\in A\mid \tau(a)=-a\}$ yields a $k$-module
  decomposition $A=B\oplus E$. The $k$-module $E$ is a $B$-bimodule,
  hence for every right ideal $I\subset B$
  \[
    I\cdot A = I\oplus I\cdot E
    \qquad\text{and}\qquad
    I=(I\cdot A)\cap B.
  \]
  If $(I_n)_{n\in\mathbb{N}}$ is a descending chain of right ideals in
  $B$, then there is an integer $N$ such that
  $I_n\cdot A=I_N\cdot A$ for all $n\geq N$,
  hence
  \[
    I_n=(I_n\cdot A)\cap B = (I_N\cdot
    A)\cap B=I_N
    \quad\text{for all $n\geq N$.}
  \]
  Thus, $B$ is right artinian.

  In order to show that $\rad(B)=0$, we show that the
  right ideal $\rad(B)\oplus \rad(B)\cdot E$ in $A$ is a
  nilideal. This ideal is therefore contained in $\rad(A)$
  by~\cite[(4.11)]{Lam}, hence it vanishes 
  since $\rad(A)=0$. This is thus
  sufficient to prove that $\rad(B)=0$.

  For the rest of the proof, fix $r\in \rad(B)$ and $s\in \rad(B)\cdot
  E$, and let 
  $M(d)$ denote the set of all products of $2d-1$ factors $r$, $s$ in
  $A$, i.e., all monomials of total degree $2d-1$ in $r$,
  $s$. By induction on $d$, we show below that
  $M(d)\subset A\cdot \rad(B)^d\cdot A$ for all
  $d\geq1$, 
  hence that $(r+s)^{2d-1}\in A\cdot
  \rad(B)^d\cdot A$. We know from the first part of the proof
  that $B$ is right artinian, hence $\rad(B)$ is nilpotent
  by~\cite[(4.12)]{Lam}. If the integer $N$ is such that
  $\rad(B)^N=0$, it will follow that $(r+s)^{2N-1}=0$ for all $r+s\in
  \rad(B)\oplus \rad(B)\cdot E$.

  For all integers $d\geq1$ and $i\in\{-d+1,\,\ldots,\,d\}$,
  let 
  \[
    M(d,i)=
    \begin{cases}
      M(d)\cap A\cdot \rad(B)^{2d-1}&\text{if $i=-d+1$},\\
    M(d)\cap A\cdot \rad(B)^{d-i}\cap
    \bigl(\bigcup_{j\in\mathbb{N}} 
    A\cdot \rad(B)^{d-1+i}\cdot E\cdot \rad(B)^j\bigr)&\text{if
      $i>-d+1$}.
    \end{cases}
  \]
  Thus, $M(1,0)=\{r\}$ and $M(1,1)=\{s\}$, and $M(1)=M(1,0)\cup
  M(1,1)$. It is clear from the definitions that $M(d,i)r^2\subset
  M(d+1,i-1)$ for all $d$, 
  $i$. Moreover, since $s\in \rad(B)\cdot E$ we have $\rad(B)^js\subset
  \rad(B)^{j+1}\cdot E$ for all $j\geq0$, and since $\tau(s)=-s$ it
  follows that $E\cdot \rad(B)^js\subset B$ for all
  $j\geq0$. Therefore,
  \[
    M(d,i)rs\subset M(d+1,2-i) \quad\text{and}\quad
    M(d,i)sr\subset M(d+1,1-i) \quad\text{for all $d$, $i$.}
  \]
  Also, from $E\cdot
  \rad(B)^js\subset B$ and $s\in \rad(B)\cdot E$ it follows that
  $s^2\in \rad(B)$ and $(E\cdot \rad(B)^js)s\subset \rad(B)\cdot E$
  for all $j\geq0$, 
  hence $M(d,i)s^2\subset M(d+1,i)$ for all $d$, $i$. Since
  $M(1)=M(1,0)\cup M(1,1)$ and
  \[
    M(d+1)=M(d)r^2\cup M(d)rs\cup M(d)sr\cup M(d)s^2,
  \]
  induction on $d$ yields
  \[
    M(d)=\bigcup_{i=-d+1}^d M(d,i) \qquad\text{for all $d\geq1$.}
  \]
  By definition, $M(d,i)\subset A\cdot \rad(B)^d\cdot
  A$, because $d-i\geq d$ if $i\leq0$ and $d-1+i\geq d$ if
  $i\geq1$. Therefore, $M(d)\subset A\cdot \rad(B)^d\cdot
  A$ and the proof of~(a) is complete.
  \medbreak

  (b) Now, assume $A$ is semilocal and write $\overline A$ for
  $A/\rad(A)$, so $\overline A$ is right artinian and $\rad(\overline
  A)=0$. For every $a\in A$, write $\overline a$ for the image 
  of $a$ in $\overline A$. The automorphism $\tau$ preserves $\rad(A)$
  and induces an automorphism $\overline\tau$ of $\overline A$. If
  $\overline\tau(\overline a) = \overline a$, then
  $\overline a = \frac12
  \overline{(a+\tau(a))}$, hence the $k$-subalgebra of $\overline
  A$ fixed under $\overline\tau$ is $B/B\cap\rad(A)$. Part~(a) of the
  proof shows that $B/B\cap\rad(A)$ is right artinian and
  $\rad(B/B\cap\rad(A)) = 0$. But $B\cap\rad(A)\subset\rad(B)$
  by~\cite[(5.6)]{Lam} and $\rad(B/B\cap\rad(A)) =
  \rad(B)/B\cap\rad(A)$ by~\cite[(4.6)]{Lam}, hence
  $\rad(B)=B\cap\rad(A)$. Therefore $B/B\cap\rad(A)=B/\rad(B)$, hence
  $B$ is semilocal.

  If $A$ is rad-adically complete, then it is clear from the
  definitions that $B$ is $(B\cap\rad(A))$-adically complete. But
  $B\cap\rad(A)=\rad(B)$, hence $B$ is rad-adically complete. This
  completes the proofs of Lemma~\ref{lem:Psplit2} and
  Proposition~\ref{prop:Psplit2}. 
\end{proof}

\begin{theorem}
  \label{thm:octasplit}
  Assume that $\End_\catC(C)$ is semilocal and rad-adically complete
  for every object $C\in\catC$. Then the octagon of Witt
  groups~\eqref{eq:octaWitt} is exact.
\end{theorem}

\begin{proof}
  By Proposition~\ref{prop:Psplit2}, the hypothesis on endomorphism
  algebras also holds for the double category $\catC'$, hence we may
  use the same shifting argument as in the proof of
  Theorem~\ref{thm:octamain} to see that it suffices to prove the
  exactness of
  \[
    \xymatrix{
      W(\catC,D,\delta) \ar[r]^-{W(\widehat F)}&
      W(\catC',D',\delta')\ar[r]^-{W(\widehat F')}&
      W(\catC'',D'',\delta'').
    }
  \]
  Theorem \ref{thm:octamain} yields $W(\widehat F')\circ W(\widehat F)
  = 0$. To prove that the kernel of $W(\widehat F')$ lies in the image
  of $W(\widehat F)$, pick a symmetric space
  $(C',h')\in\catS(\catC',D',\delta')$ such that $W(\widehat
  F')(C',h')=0$. If $(C',h')$ is split isotropic, consider a direct
  summand $L\xrightarrow{t} C'\xrightarrow{s}L$ such that $s\circ
  t=\Idfunc_L$ and ${h'}^{-1}\circ D'(t)\circ h'=0$. Substituting
  $s-\frac12s\circ {h'}^{-1}\circ D'(t\circ s)\circ h'$ for $s$, we
  may assume $s\circ {h'}^{-1}\circ D'(s)=0$. The idempotent $e=t\circ
  s\in\End_{\catC'}(C')$ then satisfies $\sigma_{h'}(e)\circ
  e=0=e\circ \sigma_{h'}(e)$, hence $e+\sigma_{h'}(e)$ also is an
  idempotent. Since $\catC'$ is pseudo-abelian we may consider a
  splitting of $\Idfunc_{C'}-e-\sigma_{h'}(e)$, i.e., morphisms
  $M\xrightarrow{t_M} C'\xrightarrow{s_M} M$ in $\catC'$ such that
  $s_M\circ t_M=\Idfunc_M$ and $t_M\circ
  s_M=\Idfunc_{C'}-e-\sigma_{h'}(e)$. Then $(t,{h'}^{-1}\circ D'(s),
  t_M) \colon L\oplus D'(L)\oplus M \to C'$ is an isomorphism with
  inverse $\Bigl(
  \begin{smallmatrix}
    s\\ D'(t)\circ h'\\ s_M
  \end{smallmatrix}
  \Bigr)$ such that
  \begin{equation}
    \label{eq:giso}
    \begin{pmatrix}
      D'(t)\\
      D'({h'}^{-1}\circ D'(s))\\
      D'(t_M)
    \end{pmatrix}
    \circ h' \circ
    \begin{pmatrix}
      t&
      {h'}^{-1}\circ D'(s)& t_M
    \end{pmatrix}
    =
    \begin{pmatrix}
      0&\Idfunc_{D'(L)}&0\\
      \delta'_{D'(L)}&0&0\\
      0&0& D'(t_M)\circ h'\circ t_M
    \end{pmatrix}.
  \end{equation}
  Since the left side is an isomorphism, it follows that $D'(t_M)\circ
  h'\circ t_M$ is an isomorphism. Letting $g=D'(t_M)\circ h'\circ
  t_M$, we have $(M,g)\in \catS(\catC',D',\delta')$, and
  \eqref{eq:giso} shows that $(t,{h'}^{-1}\circ D'(s),t_M)$ is an
  isometry
  $H(L)\perp (M,g) \simeq (C',h')$, hence $(C',h')$ is Witt-equivalent
  to $(M,g)$. If $(M,g)$ is again split isotropic, we may repeat the
  procedure to obtain another isometry $(C',h')\simeq H(L)\perp H(L_1)
  \perp (M_1,g_1)$. But since $\End_{\catC'}(C')$ is semilocal,
  \cite[Prop.~II(5.2.6)]{Knus}\footnote{Knus assumes that the category
    is abelian and the endomorphism algebra of every object is right
    artinian, but inspection of the proof shows that all the arguments
    go through without change under our weaker hypotheses.} shows that
  $C'$ has no direct sum 
  decomposition with infinitely many nonzero summands, hence the
  procedure terminates in a finite number of steps and yields a
  split anisotropic symmetric space that is Witt-equivalent to
  $(C',h')$. Therefore, we may assume $(C',h')$ is
  split anisotropic. Witt cancellation holds in
  $\catS(\catC',D',\delta')$ by \cite[Th.~II(6.6.1)]{Knus}, hence
  $W(\widehat F')(C',h')=0$ means that $\catS(\widehat F')(C',h')$ is
  hyperbolic. Proposition~\ref{prop:Psplit} then yields an isometry
  $\theta\colon (\conjt'(C'),-\alpha'_{C'}\circ \conjt'(h')) \to
  (C',h')$ such that $\theta\circ\conjt'(\theta) = \lambda'_{C'}$, and
  Lemma~\ref{lem:extend} shows that there exists a symmetric space
  $(X,h)\in\catS(\catC,D,\delta)$ such that $\catS(\widehat F)(X,h) =
  (C',h')$. Thus, the Witt class of $(C',h')$ is the image of the Witt
  class of $(X,h)$ under $W(\widehat F)$.
\end{proof}

Theorem~\ref{thm:octasplit} applies for instance in the case where
$\catC=\catP_A$, with $A$ a semilocal rad-adically complete
$k$-algebra. To see this, note that each $P\in\catP_A$ is a direct
summand in a finitely generated free $A$-module, hence
$\End_A(P)\simeq e\,M_n(A)\,e$ for some integer~$n$ and some
idempotent $e$ in the matrix algebra $M_n(A)$. If $A$ is semilocal and
rad-adically complete, then $M_n(A)$ has the same properties since
$R\bigl(M_n(A)\bigr) = M_n\bigl(\rad(A)\bigr)$,
see~\cite[p.~61 and (20.4)]{Lam}. Moreover, $e\,M_n(A)\,e$ also is
semilocal and 
rad-adically complete because $\rad(e\,M_n(A)\,e) = e\,
R\bigl(M_n(A)\bigr)\,e=R\bigl(M_n(A)\bigr)\cap (e\,M_n(A)\,e)$,
see~\cite[(21.10)]{Lam}. In particular, the octagons of Witt groups
derived from~\eqref{eq:octaherm} or~\eqref{eq:octahermGBMF} are exact
when $A$ is a right artinian $k$-algebra, hence
Theorem~\ref{thm:octasplit} extends the Grenier-Boley--Mahmoudi
results in~\cite{GBM} and Ranicki's results in~\cite[(0.2)]{HTW} for
the case where the characteristic is different from~$2$. (For
$A$ right artinian and semisimple, the quadratic $L$-groups
$L_n(A,\sigma,\varepsilon)$ with $n$ even are Witt groups,
see~\cite[\S1]{R}. Note that Ranicki does not assume $2$ is
invertible.)

\subsection{Abelian categories}
\label{subsec:abel}

In this subsection, we assume that $\catC$ is an (essentially small)
abelian category, which we endow with its usual exact structure, and
we only consider dualities $(D,\delta)$ with $D$ exact. For any
symmetric space $(C,h)\in\catS(\catC,D,\delta)$ and any subobject
$\varphi\colon L\hookrightarrow C$, the \emph{orthogonal subobject} is
defined as $L^\perp:= \ker(D(\varphi)\circ h)$. The space $(C,h)$ is
said to be \emph{isotropic} if there is a nonzero subobject $L\subset
C$ such that $L\subset L^\perp$, and \emph{anisotropic}
otherwise. (Note that $L$ is not required to be a direct summand.) The
space $(C,h)$ is said to be \emph{metabolic} (\emph{weakly hyperbolic}
in~\cite[Def.~7.5.4]{Scharlau}) if there is a subobject $L\subset C$
such that $L=L^\perp$. The resulting Witt group $W(\catC,D,\delta)$ is
called \emph{weak Witt group} in~\cite[Def.~7.5.4]{Scharlau}.

In every conjugation $(\conjt,\lambda)$ we consider, the functor
$\conjt$ is assumed to be exact. The double category $\catC'$ is then
abelian, and the functors $F$, $G$, $\Phi$, $\Psi$ are
exact. Therefore, given a conjugation $(\conjt,\lambda)$ and a
commuting duality $(D,\delta)$, the functors in the
octagon~\eqref{eq:octasspaces} induce Witt group homomorphisms, which
again yield the $8$-periodic diagram~\eqref{eq:octaWitt}. Hyperbolic
spaces are clearly metabolic, hence Theorem~\ref{thm:octamain} again
shows that this diagram is a chain complex. The goal of this
subsection is to show that the homology of this chain complex vanishes
when $\catC$ is an artinian category.

\begin{lemma}
  \label{lem:stronghyp}
  Property~(P) holds for anisotropic spaces in $\catS(\catC,D,\delta)$.
\end{lemma}

\begin{proof}
  Let $(C,h)\in\catS(\catC,D,\delta)$ be an anisotropic symmetric space.
  Let $\catS(\widehat F)(C,h)=(F(C),h')$ with $h'=\widehat\phi_C\circ
  F(h)$ and let $\varphi'\colon L'\hookrightarrow F(C)$ be a subobject
  such that $L'={L'}^\perp$. Write 
  $\varphi=G(\varphi')$ and $L=G(L')$, so $\varphi\colon L\to
  C\oplus\conjt(C)$ is a morphism in $\catC$. Write also
  $\epsilon_1$, $\epsilon_2$, $\pi_1$, $\pi_2$ for the canonical
  morphisms
  \[
    \epsilon_1\colon C\to C\oplus\conjt(C),\quad
    \epsilon_2\colon \conjt(C)\to C\oplus\conjt(C),\quad
    \pi_1\colon C\oplus\conjt(C)\to C,\quad
    \pi_2\colon C\oplus\conjt(C)\to\conjt(C).
  \]
  \textit{Claim: $\pi_2\circ\varphi\colon L\to\conjt(C)$ is an
    isomorphism.}

  We show that $\pi_2\circ\varphi$ is a monomorphism and an
  epimorphism. If $\psi\colon M\to L$ is a morphism such that
  $\pi_2\circ\varphi\circ\psi=0$, then $\varphi\circ\psi$ factors
  through the kernel of $\pi_2$, which is $C$, hence there exists
  $\psi_1\colon M\to C$ such that
  $\varphi\circ\psi=\epsilon_1\circ\psi_1$. Since $D'(\varphi')\circ
  h'\circ \varphi'=0$, it follows that
  \[
    D(\psi_1)\circ D(\epsilon_1)\circ
    \begin{pmatrix}
      h&0\\0&-\alpha_C\circ\conjt(h)
    \end{pmatrix}
    \circ\epsilon_1\circ\psi_1=D(\varphi\circ\psi)\circ G(h')\circ
    \varphi\circ\psi=0.
  \]
  But $D(\epsilon_1)\circ G(h')\circ\epsilon_1=h$ and $(C,h)$ is
  anisotropic, hence $\psi_1=0$ and
  $\varphi\circ\psi=0$. But $\varphi$ is a monomorphism, hence
  $\psi=0$. This shows that $\pi_2\circ\psi$ is a 
  monomorphism.

  Similarly, if $\psi\colon M\to D\conjt(C)$ is a morphism such that
  $D(\pi_2\circ\varphi)\circ\psi=0$, then $G(h')^{-1}\circ
  D(\pi_2)\circ\psi$ factors through $\ker(D(\varphi)\circ G(h'))=L$,
  hence $\psi=0$ because $(\conjt(C),-\alpha_C\circ\conjt(h))$ is
  anisotropic and $D(\pi_2)$ is a monomorphism. Therefore,
  $D(\pi_2\circ\varphi)$ is a monomorphism, hence $\pi_2\circ\varphi$
  is an epimorphism because $D$ is an exact functor.
  \medbreak

  The claim is thus proved, and me may consider
  $\theta=(\pi_1\circ\varphi) \circ(\pi_2\circ\varphi)^{-1}\colon
  \conjt(C)\to C$. Let $L'=(L,\ell)$ for some $\ell\colon\conjt(L)\to
  L$. Since $\varphi'$ is a morphism in $\catC'$, there are
  commutative diagrams
  \[
    \xymatrix{
      \conjt(L)\ar[r]^-{\conjt(\varphi)} \ar[d]_{\ell} &
      \conjt(C)\oplus\conjt\conjt(C) \ar[r]^-{\conjt(\pi_1)}
      \ar[d]_{\Lambda_C} &
      \conjt(C) \ar@{=}[d]\\
      L\ar[r]^-{\varphi}& C\oplus\conjt(C) \ar[r]^-{\pi_2}&
      \conjt(C)
    }
    \quad\text{and}\quad
    \xymatrix{
      \conjt(L)\ar[r]^-{\conjt(\varphi)} \ar[d]_{\ell} &
      \conjt(C)\oplus\conjt\conjt(C) \ar[r]^-{\conjt(\pi_2)}
      \ar[d]_{\Lambda_C} &
      \conjt\conjt(C) \ar[d]^{\lambda_C}\\ 
      L\ar[r]^-{\varphi}& C\oplus\conjt(C) \ar[r]^-{\pi_1}&
      C
    }
  \]
  It follows that $\ell=(\pi_2\circ\varphi)^{-1}\circ
  \conjt(\pi_1\circ\varphi)$ and $(\pi_1\circ\varphi)\circ\ell =
  \lambda_C\circ\conjt(\pi_2\circ\varphi)$, hence
  $\theta\circ\conjt(\theta) = \lambda_C$. Since $\lambda_C$ is an
  isomorphism, we see that $\theta$ is invertible on the right. But
  taking the image of each side under $\conjt$ and using
  $\lambda_C\circ\conjt\conjt(\theta)=\theta\circ\conjt(\lambda_C)$ we
  obtain
  $\conjt(\theta)\circ\lambda_C^{-1}\circ\theta=\Idfunc_{\conjt(C)}$,
  hence $\theta$ is also invertible on the left and is therefore an
  isomorphism.

  To compute $D(\theta)\circ h\circ\theta$, we use
  \[
    G(h')= D(\pi_1)\circ h\circ \pi_1 - D(\pi_2)\circ(\alpha_C\circ
    \conjt(h)) \circ\pi_2.
  \]
  Since $D(\varphi)\circ G(h')\circ\varphi=0$, this equation yields
  \[
    D(\pi_1\circ\varphi)\circ h \circ (\pi_1\circ\varphi) =
    D(\pi_2\circ\varphi)\circ \alpha_C\circ\conjt(h) \circ
    (\pi_2\circ\varphi),
  \]
  hence $D(\theta)\circ h\circ\theta = \alpha_C\circ\conjt(h)$.
\end{proof}

Recall that an abelian category $\catC$ is said to be \emph{artinian}
(resp.\ \emph{noetherian}) if the descending (resp.\ ascending) chain
condition for 
subobjects holds for every $C\in\catC$. These properties are inherited
by the double category when $\catC$ carries a conjugation, since the
forgetful functor $G$ carries subobjects to subobjects. If $\catC$
carries a duality, the artinian and the noetherian conditions are
equivalent, and imply that each object has finite length.

\begin{theorem}
  \label{thm:abelcat}
  If $\catC$ is an artinian abelian category, the
  octagon of Witt groups~\eqref{eq:octaWitt} is exact.
\end{theorem}

\begin{proof}
  By the same shifting argument
  as in the proof of Theorem~\ref{thm:octamain}, we are reduced to
  proving the exactness of the sequence
  \[
  \xymatrix{
    W(\catC,D,\delta)\ar[r]^-{W(\widehat F)}&
    W(\catC',D',\delta') \ar[r]^-{W(\widehat F')}&
    W(\catC'',D'',\delta'').
  }
\]
We already know from Theorem~\ref{thm:octamain} that this is a
zero-sequence. Let $(C',h')\in\catS(\catC',D',\delta')$ be a symmetric
space whose Witt equivalence class is in the kernel of $W(\widehat
F')$. Since $\catC'$ is artinian, isotropic reduction (also known as
sublagrangian reduction~\cite[\S1.2.5]{Balmer}) yields an anisotropic
symmetric space Witt-equivalent to $(C',h')$, see~\cite[Cor.~2.12]{SW}
or \cite[Th.~7.5.5]{Scharlau}. We may therefore assume $(C',h')$ is
anisotropic. Since the Witt class of $(C',h')$ lies in the kernel of
$W(\widehat F')$, it follows from~\cite[Th.~(4.9)]{Youssin} that
$\catS(\widehat F')(C',h')$ is metabolic. Lemma~\ref{lem:stronghyp}
shows that $(C',h')$ satisfies Property~(P), hence
Lemma~\ref{lem:extend} yields a symmetric space
$(X,h)\in\catS(\catC,D,\delta)$ whose Witt equivalence class is mapped
under $W(\widehat F)$ to the Witt equivalence class of $(C',h')$.
\end{proof}

\subsection{Example: The Warshauer octagon}
\label{subsec:exdual2}
Let $R=k[t,t^{-1}]$ be the ring of Laurent
polynomials in one indeterminate over a field $k$ of characteristic
different from~$2$, and let $\catF_R$ be the category of right
$R$-modules of finite length, i.e., of finite-dimensional $k$-vector
spaces with an invertible operator (given by the action of $t$).
Consider the structure $(\Idfunc_R,t)$ on $R$ (in the sense of
\S\ref{subsec:exconjug1}); the corresponding 
$(\Idfunc_R,t)$-twisted quadratic extension of $R$ is $R'=k[u,u^{-1}]$
where $u^2=t$, and the conjugation $(\conjt,\lambda)$ on $\catF_R$
derived from $(\Idfunc_R,t)$ has $\conjt=\Idfunc_{\catF_R}$ and
$\lambda_X\colon X\to X$ is multiplication by $t$ for
$X\in\catF_R$. As in~\S\ref{subsec:exconjug1}, we endow $R'$ with the
structure $(\conjt',-t)$, where $\conjt'(u)=-u$, and identify
$\catF_R'=\catF_{R'}$. 

To define a duality on $\catF_R$, fix an element $\mu\in k^\times$ and
let 
\[
  D(X)=\Hom_k(X,k)\qquad\text{for $X\in\catF_R$},
\]
endowed with the $R$-module action defined by the condition
\[
  \langle\pi t,x\rangle = \langle \pi, xt^{-1}\rangle\mu^2
  \qquad\text{for $\pi\in D(X)$ and $x\in X$,}
\]
where $\langle\text{\textvisiblespace}\,,
\text{\textvisiblespace}\rangle$ is the canonical pairing
$\Hom_k(X,k)\times X\to k$. With $\delta\colon \Idfunc_{\catF_R}
\natiso DD$ the usual natural isomorphism given by $\langle
\delta_X(x),\pi\rangle =\langle\pi,x\rangle$ for $x\in X$ and $\pi\in
D(X)$, the pair $(D,\delta)$ is a duality\footnote{The usual duality
  on the category of modules of finite length is different, see
  \cite[Ch.~6, \S1]{Scharlau}.} on
$\catF_R$. It is 
commuting with the conjugation $(\Idfunc_{\catF_R},\lambda)$ 
under the natural isomorphism $\alpha\colon D\natiso D$ given by
\[
  \alpha_X\colon D(X)\to D(X), \qquad \pi\mapsto \pi\mu \text{ for
    $\pi\in D(X)$}.
\]
The construction in Proposition~\ref{prop:altdual} yields a new
duality $(D,\delta_\alpha)$ where
\[
  \langle (\delta_\alpha)_X(x),\pi\rangle = \langle\pi, xt^{-1}\rangle
  \mu \qquad\text{for $x\in X$ and $\pi\in D(X)$.}
\]
On the other hand, the duality $(D',\delta')$ on $\catF_{R'}$ is given
by 
$D'(X')=\Hom_k(X',k)$ with the $R'$-module action given by
\[
  \langle\pi u,x\rangle = -\langle\pi,xu^{-1}\rangle \mu
  \qquad\text{for $\pi\in D'(X')$ and $x\in X'$,}
\]
and $\delta'_{X'} = \delta_{G(X')}$ for $X'=(X,f)\in\catF'_R$.
Since the automorphism $\conjt'$ on $R'$ maps $u$ to $-u$, for
$X'\in\catF_{R'}$ the
$R'$-module action on
$\conjt'D'(X')=\Hom_k(X',k)$ is given by
\[
  \langle\pi u,x\rangle = \langle\pi,xu^{-1}\rangle \mu
  \qquad\text{for $\pi\in D'(X')$ and $x\in X'$.}
\]
Moreover, $\alpha'_{X'}=-\alpha_{G(X')}$ and
$\lambda'_{X'}=-\lambda_{G(X')}$, hence the duality
$(\conjt'D',\delta'_{\alpha'})$ satisfies
\[
  \langle(\delta'_{\alpha'})_{X'}(x),\,\pi\rangle =
  \langle\pi,\,xt^{-1}\rangle\mu \qquad\text{for $\pi\in\conjt'D'(X')$
    and $x\in X'$.}
\]

For $(X,h)\in\catS(\catF_R,D,\delta)$, define
\[
  b\colon X\times X\to k\qquad
  \text{by}\quad
  b(x,y)=\langle h(x),y\rangle
  \quad\text{for $x$, $y\in X$.}
\]
The map $b$ is a nonsingular $k$-bilinear form and the
$R$-linearity of $h$ means that
\[
  b(xt,\,yt) = b(x,\,y)\mu^2
  \qquad\text{for $x$, $y\in X$.}
\]
Thus, multiplication by $t$ is a similitude with multiplier $\mu^2$.
Writing $\catB^{\pm}(k,\mu^2)$ for the category of triples
$(X,b,\ell)$ where $X$ is a finite-dimensional $k$-vector space,
$b\colon X\times X\to k$ is a nonsingular $(\pm1)$-symmetric bilinear
form and $\ell\in\End_k(X)$ is a similitude with multiplier $\mu^2$,
we thus have
\begin{equation}
  \label{eq:SFR1}
  \catS(\catF_R,D,\pm\delta)=\catB^{\pm}(k,\mu^2).
\end{equation}
When $\mu=\pm1$, the objects in this category are the inner product
spaces with isometry considered by Milnor~\cite{Milnor}.

With the commutation $\alpha$ between $(\Idfunc_{\catF_R},\lambda)$
and $(D,\delta)$ given above we may similarly
define for 
every $(X,h)\in\catS(\catF_R,D,\pm\delta_\alpha)$ the map
\[
  b\colon X\times X\to k
  \qquad\text{by}\quad
  b(x,y)=\langle h(x),y\rangle
  \quad\text{for $x$, $y\in X$.}
\]
Now, the form $b$ is $k$-bilinear but not symmetric: the
property $h=D(h)\circ(\delta_\alpha)_X$ yields
\[
  b(xt,\,y) =\pm b(y,\,x)\mu
  \qquad\text{for all $x$, $y\in X$,}
\]
from which it follows that multiplication by $t$ is a similitude with
multiplier $\mu^2$.
To obtain an alternative description of this category, consider the
category $\catA(k)$ whose objects are pairs $(X,a)$ where $X$ is a
finite-dimensional $k$-vector space and $a\colon X\times X\to k$ is a
nonsingular bilinear form without symmetry requirement, and whose
morphisms are isometries. (The objects in $\catA(k)$ are the
\emph{asymmetric inner product spaces over $k$} in Warshauer's
terminology from~\cite{War}.) Recall from~\cite[Cor.~I.4.2]{War} that
to each $(X,a)\in\catA(k)$ is attached a bijective \emph{symmetry
  operator} $s\in\End_k(X)$ defined by the property that
$a(x,y)=a(y,s(x))$ for $x$, $y\in X$. There are isomorphisms of
categories
\begin{equation}
  \label{eq:SFR2}
  \catS(\catF_R,D,\delta_\alpha) = \catA(k)
  \qquad\text{and}\qquad
  \catS(\catF_R,D,-\delta_\alpha) = \catA(k)
\end{equation}
under which $(X,h)\in\catS(\catF_R,D,\pm\delta_\alpha)$ is identified
with $(X,b)\in\catA(k)$. Conversely, for each
$(X,a)\in\catA(k)$ we define $h\colon X\to D(X)$ by
$h(x)=a(x,\text{\textvisiblespace})$ for $x\in X$ and identify $(X,a)$
with $(X,h)\in\catS(\catF_R,D,\pm\delta_\alpha)$ in which the action
of $t$ is given by $xt=\pm s^{-1}(x)\mu$ for $x\in X$.

Under the identification $\catF'_R=\catF_{R'}$, and with the duality
$(D',\delta')$ defined above, we may identify by
the same construction as in~\eqref{eq:SFR1}
\begin{equation}
  \label{eq:SFR3}
  \catS(\catF_{R'},D',\pm\delta')=\catB^{\pm}(k,-\mu).
\end{equation}

For every symmetric space
$(X',h')\in\catS(\catF_{R'},\conjt'D',\pm\delta'_{\alpha'})$ the
associated $k$-bilinear map
$b'\colon X'\times X'\to k$
defined by
$b'(x,y)=\langle h'(x),\,y\rangle$
for $x$, $y\in X'$
satisfies
\[
  b'(xu,yu)=b'(x,y)\mu
  \qquad\text{and}\qquad
  b'(xt,y) = \pm b'(y,x)\mu
  \quad\text{for $x$, $y\in X'$.}
\]
Since $u^2=t$, these equations imply
\[
  b'(yu,x) = \pm b'(xt,yu)\mu^{-1} =
  \pm b'(xu,y)
  \qquad\text{for $x$, $y\in X'$,}
\]
hence the map
$b\colon X'\times X'\to k$ defined by
$b(x,y)=b'(xu,y)$ for $x$, $y\in X'$
is a nonsingular $(\pm1)$-symmetric bilinear form on $X'$ such that
$b(xu,yu) = b(x,y) 
\mu$ for $x$, $y\in X'$. The map $h'\colon X'\to \conjt'D'(X')$ is
uniquely determined by $b$ since $\langle h'(x),y\rangle =
b(xu^{-1},y)$ for $x$, $y\in X'$. Therefore, by mapping $(X',h')$ to
$(X',b)$ we may identify
\begin{equation}
  \label{eq:SFR4}
  \catS(\catF_{R'},\conjt'D',\delta'_{\alpha'}) = \catB^{\pm}(k,\mu).
\end{equation}

With the identifications~\eqref{eq:SFR1} and \eqref{eq:SFR3} above,
the functor $\catS(\widehat F)\colon\catB^{\pm}(k,\mu^2) \to
\catB^{\pm}(k,-\mu)$ is given by
\[
  \catS(\widehat F)(X,b,\ell) = (X\oplus X, b\perp\langle-\mu\rangle
  b, \ell'), \quad\text{where $\ell'(x_1,x_2)=(\ell(x_2),x_1)$ for
    $x_1$, $x_2\in X$.}
\]
It gives rise to the maps denoted by $I_{\pm1}$ in~\cite[p.~109]{War}.

The functor $\catS(\widecheck F)\colon\catA(k)\to\catB^\pm(k,\mu)$
takes two different forms under the identifications~\eqref{eq:SFR2}
and \eqref{eq:SFR4}, depending on whether $\catA(k)$ is viewed as
$\catS(\catF_R,D,\delta_\alpha)$ or $\catS(\catF_R,D,-\delta_\alpha)$:
for $(X,a)\in\catA(k)$,
\[
  \catS(\widecheck F)_{\pm}(X,a) = (X\oplus X,a_\pm,\ell_\pm)
\]
where for $x_1$, $x_2$, $y_1$, $y_2\in X$,
\[
  a_\pm\bigl((x_1,x_2),(y_1,y_2)\bigr) = \mu\bigl(a(x_1,y_2) \pm
  a(y_1,x_2)\bigr) \quad\text{and}\quad \ell_\pm(x_1,x_2) = (\pm
  s^{-1}(x_2)\mu, x_1).
\]
These functors give rise to the maps $m_{\pm1}$ in~\cite[p.~108]{War}
up to similitude: the definition of $a_\pm$ in~\cite[p.~108]{War}
lacks the factor $\mu$, which does not affect the main results.

By contrast, the functor $\catS(\widehat G)\colon \catB^\pm(k,-\mu)
\to \catA(k)$ has a uniform description: for
$(X,b',\ell')\in\catB^\pm(k,-\mu)$,
\[
  \catS(\widehat G)(X,b',\ell') = (X,b)
  \qquad\text{with\quad $b(x,y)=\mu^{-1}b'(x,\ell(y))$ for $x$, $y\in X$.}
\]
It yields the maps $d_{\pm1}$ in~\cite[p.~109]{War}. Finally, the
functor $\catS(\widecheck G)\colon \catB^\pm(k,\mu) \to
\catB^\pm(k,\mu^2)$ yields the ``squaring map'' of~\cite[p.~108]{War}:
for $(X,b,\ell)\in\catB^\pm(k,\mu)$,
\[
  \catS(\widecheck G)(X,b,\ell) = (X,b,\ell^2).
\]

The octagon~\eqref{eq:octasspaces} takes the form
\begin{equation*}
  \begin{split}
  \xymatrix{
    \catB^+(k,\mu^2)\ar[r]^-{\catS(\widehat F)}&
    \catB^+(k,-\mu)\ar[r]^-{\catS(\widehat G)}&
    \catA(k)\ar[d]^{\catS(\widecheck F)_-}
    \\
    \catB^+(k,\mu)\ar[u]^{\catS(\widecheck G)}
    &
    &
    \catB^-(k,\mu)
    \ar[d]^{\catS(\widecheck G)}
    \\
    \catA(k)\ar[u]^{\catS(\widecheck F)_+}
    &
    \catB^-(k,-\mu)\ar[l]_-{\catS(\widehat G)}
    &
    \catB^-(k,\mu^2)\ar[l]_-{\catS(\widehat F)}
  }
  \end{split}
\end{equation*}
Theorem~\ref{thm:abelcat} shows that the corresponding octagon of Witt
groups is exact.

\bibliographystyle{amsalpha}
\bibliography{Octagons}

\end{document}